\documentclass[11pt]{article}         
\usepackage{amsmath,amsthm,amsfonts,graphicx} 
\usepackage{mathtools}
\usepackage{multirow}
\usepackage{xcolor}
\def\smallddots{\mathinner{\raise7pt\hbox{.}\raise4pt\hbox{.}\raise1pt\hbox{.}}} 
\def\smallsdots{\mathinner{\raise1pt\hbox{.}\raise4pt\hbox{.}\raise7pt\hbox{.}}}

\DeclareMathOperator{\diag}{diag}
\DeclareMathOperator{\prob}{Probability}
\DeclareMathOperator{\rank}{rank}
\DeclareMathOperator{\nrank}{nrank}
\newtheorem{theorem}{Theorem}[section]

\numberwithin{equation}{section}
\numberwithin{table}{section}
\newtheorem{lemma}{Lemma}[section]
 
\newtheorem{corollary}{Corollary}[section]

\newtheorem{algorithm}{Algorithm}[section]
\newtheorem{example}{Example}[section]
\newtheorem{definition}{Definition}[section]
\newtheorem{assumption}{Assumption}[section]

\newtheorem{remark}{Remark}[section]

\begin{document}
 
\title{Low Rank Approximation at Sublinear Cost\footnote{Some results of this paper have been presented at INdAM (Istituto Nazionale di Alta Matematica) Meeting "Structured Matrices in Numerical Linear Algebra: Analysis, Algorithms and Applications", Cortona, Italy, September 4-8, 2017, and  at MMMA, Moscow, Russia, August 2019.}
} 
\author{Qi Luan$^{[1],[a]}$, Victor Y. Pan$^{[1, 2],[b]}$, 
John Svadlenka
$^{[1],[c]}$, 
and  Liang Zhao$^{[1, 2],[d]}$
\\
\\
%\and\\ 
$^{[]}$ Ph.D. Programs in  Computer Science and Mathematics \\
The Graduate Center of the City University of New York \\
New York, NY 10036 USA \\
$^{[2]}$ Department of Computer Science \\
Lehman College of the City University of New York \\
Bronx, NY 10468 USA 
\\
$^{[a]}$ qi\_luan@yahoo.com \\
$^{[b]}$ victor.pan@lehman.cuny.edu \\ 
http://comet.lehman.cuny.edu/vpan/  \\
$^{[c]}$ jsvadlenka@gradcenter.cuny.edu \\ 
$^{[d]}$ Liang.Zhao1@lehman.cuny.edu \\
} 
\date{}
 
\maketitle

%------------------------------------------------------------------------------
%------------------------------------------------------------------------------
 
\begin{abstract}  
Low Rank Approximation {(\em LRA)} of a matrix is a hot research subject, fundamental for Matrix and Tensor Computations and Big Data Mining and Analysis. Computations with low rank matrices can be performed 
 at sublinear cost -- by using much fewer floating-point   operations (flops)
 than an input matrix has entries, but can we  compute  {\em LRA}  at sublinear
 cost?  
 This is routinely done in 
computational practice for a large class of inputs, even though  any sublinear cost LRA algorithm  fails most miserably on  worst case matrices.

To provide insight into this controversy we first accelerate some popular near-optimal
random sketching
 {\em LRA} algorithms -- to run them at sublinear cost. Then  we define two probabilistic structures in the space of input matrices and 
 estimate that the expected spectral and Frobenius  error norms
 for the output {\em LRA} of the accelerated algorithms 
 stay within a reasonable factor from their optima under both models,  and so  these sublinear cost algorithms  only fail
for a very narrow input class.

Our upper estimates for their output accuracy 
are still quite high,  but under some additional semi-heuristic amendments 
 the algorithms have consistently 
output
 accurate {\em LRA} of various synthetic and real-world matrices
 in our numerical tests.
 \end{abstract}
  

% - - - - - - - - - - - - - - - - - 

 
%--------------------------------------------------------------------------
 
\section{Introduction}\label{sintro}
  
% - - - - - - - - - - - - - - - - - - - - - - - - - - - - - - - - - - - - -

%\subsection
{\bf 1.1. Background for {\em LRA}.}   
%\label{sprbpr} 
Low rank approximation ({\em LRA}) of a matrix
 is a hot research area of Numerical Linear  Algebra  and Computer Science 
 with applications to fundamental matrix and tensor computations
 and  data mining and analysis (see
the  surveys \cite{M11,HMT11,
TYUC17,MT20,
TWa},   
  and pointers to the huge bibliography
  therein). %Next  we outline our study of {\em LRA}, to be formalized later.\footnote{It is customary in this field to rely on a very informal basic concept of ow rank approximation; in  high level description of our {\em LRA} algorithms we also use some other informal concepts such as ``large",  ``small",  ``ill-" and ``well-conditioned",  ``near", ``close", and ``approximate", quantified in context; later we complement this description with formal presentation and analysis.}
 
%\cite{GZT97}, \cite{GTZ97}, \cite{T00}, 
%\cite{FKV}, and \cite{DKM06},
%for sample early works. 
 The size of the matrices defining 
 Big Data (e.g., unfolding matrices of multidimensional tensors) 
  is frequently so large that only a small fraction of all their entries  fits primary computer memory, although quite typically these 
  matrices admit their {\em LRA}
  \cite{UT19}. One can operate with low rank matrices at {\em  sublinear  computational cost} -- by using much fewer flops  than an input matrix has entries, but can we compute {\em LRA} at sublinear cost? 
   Yes and no.
   
No, because  any  sublinear cost {\em LRA}  {\em fails most miserably} on worst case inputs and even on the matrices of  the small families of Example \ref{exdlt} of our Appendix  \ref{shrdin}. 
    
   Yes, because  Cross-Appro\-xi\-ma\-tion {\em (C-A)}
iterations
(see \cite{OZ18,
 ALS24},
 and the  bibliography therein),
 adapting to {\em LRA} the   Alternating Direction
Implicit {\em (ADI)} method \cite{S16}, have been 
running at sublinear cost and consistently computing close {\em LRA} of a large and important class of matrices.  

Nontrivial  progress  made so far in the analysis of these iterations (see \cite{CD13,
OZ18,CK20,
ALS24,CY25}) only partly 
 explains this 
  ``yes" and ``no"
 coexistence.
\\   
{\bf 1.2. Dual sketching  algorithms.} 
 To provide a new insight into it, we
first recall  {\em sketching {\em LRA} algorithms}
\cite{M11,
HMT11,W14,
TYUC17,MT20,
TWa}. With a high probability   ({\em whp}) they compute
     a near-optimal  rank-$r$ approximation of an 
$m\times n$ matrix $M$ by using auxiliary  matrices
({\em sketches}) $FM$ and/or $MH$  with $k\times m$ and/or $n\times l$
random  {\em sketch matrices} 
$F$ and/or $H$, respectively, for $r\le \min\{k,l\}$, $\max\{k,l\}=r+p$, and an {\em oversampling parameter} $p>0$. 
For small positive integers $k$ and $l$ the algorithms
run at sublinear cost, except for the  computation of the sketches,  and our next goal is to run  that stage at sublinear cost as well.
One can explore a variety of ways towards this goal; to simplify our presentation, we will work with $LRA$ that only uses sketches $MH$, as in 
\cite{HMT11}, and in Sec. \ref{simppre}
 show extension of our results  to   algorithms
using both sketches $FM$ and $MH$.

Maybe the simplest way to our goal is
to {\em randomly sample} $l$  columns of the matrix  $M$ or equivalently  to let $H$ be 
the  $n\times l$ leftmost submatrix 
of a random $n\times n$ permutation matrix
under a fixed randomization model.  Clearly, in this case only $nl$ elements of $M$ are accessed, and
the calculation itself is given "for free".
We can alternatively  compute sketches $MH$ at sublinear cost by choosing
{\em Ultrasparse sketch matrices} $H$.
 Example \ref{exdlt}
tells us that these ways (as well as   any other sublinear way) cannot yield accurate {\em LRA} for all inputs $M$, but in our  
  numerical experiments with  Ultrasparse sketch matrices  we have quite consistently output close {\em LRA}. 

Towards explaining these test  results 
and  constructing a theory,
we introduce a probabilistic structure in the  space of input
matrices $M$, thus defining {\em dual 
sketching} in comparison to customary primal one, where this matrix is fixed.

The output error norm of  sketching algorithms of \cite{HMT11} is essentially defined by the product $V_1^*H$ 
where $V_1$ is the matrix 
of the $r$ top left singular vectors of $M$, associated with its $r$ top (largest)
singular values.
Now
assume 
that $V_1^*$ is the Q factor in QR or QRP factorization of a Gaussian matrix,
 $H$ has orthonormal columns, and $l-r$
 is reasonably large. Then we
 prove that whp such a sketching algorithm outputs	{\em LRA} whose 
 both Frobenius and spectral error norms 
 are within a factor of 
 $\sqrt {1 + 16~n/l}$ from their optima; this is  essentially Thm. \ref{eqtherr} 
-- our {\bf  Main result 1!}
%%VYP

While this choice of probabilistic structure is  most relevant 
to the estimation of the output errors of a sketching algorithm,
its  relevance to   {\em LRA} computation in the  
real world is debatable, as probably any choice of probabilistic structure in the space of input matrices for {\em LRA}.
 
In Sec. \ref{serrrang}
we estimate the errors of {\em LRA}  by means of dual sketching under another
model, possibly more natural:
  we fix a sketch matrix $H$
with orthonormal columns 
and consider an $m \times n$ input matrix $M = A\Sigma B +E$ for  a perturbation matrix  $E$  and for
$A\Sigma B$ being a {\em random pseudo  SVD} of a rank-$r$ matrix of size $m\times n$. Namely, we let $\Sigma$ be an $r\times r$ diagonal matrix
 with $r$ positive diagonal entries
 (as in SVD) and let $A$ and $B$ be  {\em Gaussian} matrices,
 that is, matrices 
 filled with $(m+n)r$ independent 
standard
Gaussian (normal) random variables, 
 rather than with singular 
vectors of SVD.\footnote{By saying {\em ``LRA"} we assume that $r\ll\min\{m,n\}$  and then  motivate our definition of random pseudo SVD by recalling
(e.g., from  
 \cite[Thm. 7.3]{E89} or 
\cite{RV09})
%%VQ4/16
that $\frac{1}{\sqrt k}G$ is 
close to a matrix having orthonormal columns whp for $r\ll k$ and 
a $k\times r$ Gaussian
matrix $G$. See another motivation of independent interest in Remark \ref{reprprp} 
 of Appendix \ref{srndoprpr}.}

We call such a matrix $M$ a perturbation 
 of two-sided factor-Gaussian (with expected rank 
$r$), but most of our study 
(including Thm. \ref{eqtherr} and our main result, stated below)
%VP0421
apply to a more general class of perturbed 
right factor-Gaussian matrices (with expected rank $r$) of the form
$AG+E$ where $G$ is an $r\times n$ Gaussian matrix and  $A$ is an $m \times r$ matrix of full rank $r$.

We need a longer probabilistic 
study but finally prove 
 under a rather mild assumption about the spectrum of the singular values of 
an input matrix $M$
that our sublinear cost sketching algorithm outputs a close
rank-$r$ approximations whp on the defined probability space: namely, our  output error bound  only increases
from a factor of $\sqrt {1 + 16~n/l}$  versus the optimum in Thm. \ref{eqtherr}
 to  a factor of $\sqrt {1 + 100~n/l}$, and this is  essentially Thm. \ref{therrfctr}
-- our \textbf{Main result 2!}

Both Thms.  \ref{eqtherr} and \ref{therrfctr} apply to approximation by matrices  of any fixed rank $r$ and cover  random sketching algorithms that in the case of  proper Ultrasparse sketch matrices $F$ and/or
$H$ use $O((m+n)r^2)$
flops,  running at sublinear cost 
where $r^2\ll\max\{m,n\}$ but involving {\em sublinear memory space}, that is,  much fewer than $mn$ entries of $M$ and other scalars,
already where $r\ll \min\{m,n\}$.
  \\
{\bf 1.3. Quality of dual {\em LRA}.}  
 Thms. \ref{eqtherr} and \ref{therrfctr} 
should be of some qualitative interest, e.g., they
imply that
miserable failure of the dual sketching algorithms    running at sublinear cost is highly unlikely
under  both of our probabilistic models, even though it occurs for worst case inputs.

On the other hand, the 
upper bounds of both Thms. \ref{eqtherr}
 and \ref{therrfctr}  
on the output error norm of {\em LRA} are too high to support practical use of these algorithms. We, however, propose and successfully test some semi-heuristic recipes for empirical decrease of these bounds (see Sec. \ref{srndsmpl}):	

(i) Thms. \ref{eqtherr} and \ref{therrfctr} 
hold for any matrix $H$ that has orthonormal
columns or is just well-conditioned.
We hope that the classes of  inputs $M$  that are hard for {\em LRA} vary
when matrices $H$ vary, and if so, we can widen  our success for accurate {\em LRA} by applying a sketching algorithm for a number of
distinct sketch matrices $H$ or their  combinations.
    
(ii) With sketch matrices obtained with  {\em sparse subspace embedding} \cite{C16,CDDRa,CFSa},
\cite[Sec. 3.3]{TYUC19},
\cite[Sec. 9]{MT20}, one can devise {\em LRA} algorithms running 
at linear computational cost.
According to \cite{L09},
 such acceleration  of {\em LRA}
  that run at   superlinear cost  with Gaussian, SRHT, and SRFT sketch matrices tends to
 make the  output accuracy of the {\em LRA} somewhat less reliable, although \cite{CFSa} partly overcomes
 this  problem for incoherent matrices,  that
is, filled with entries of comparable magnitude (see \cite{CFSa} or
 \cite [Def. 1.1]{CD13} for  formal definition). One can  multiply a  matrix by Subsampled Randomized Hadamard or Fourier Transform ({\em SRHT or SRFT)} dense matrices at  superlinear cost
 towards making $M$ incoherent \cite{CFSa}.
For a compromise, we devise {\em Abridged SRHT and SRFT sketch matrices}. They are  Ultrasparse, can  be multiplied by a dense matrix at sublinear cost, and
 have orthonormal columns, which allows us to apply Thms.  \ref{eqtherr} and \ref{therrfctr}.
This recipe is also indirectly supported by
 Remark \ref{reprprp} in
Appendix \ref{srndoprpr}. 
\\ 
{\bf 1.4. Related work.} 
Sublinear cost algorithms, called superfast,  have been studied intensively for
Toeplitz, Hankel, Vandermonde, Cauchy, and other structured matrices having small displacement rank  and defined by small number of parameters
(see \cite{P01,
XXG12,P15} 
and
extensive bibliography therein).
More recently, randomized {\em LRA} algorithms
running at sublinear cost
have been proposed in
\cite{MW17,
BW18,
CETW23}  for some special but large and important classes of matrices
defined by large numbers of independent parameters. Most notably, the authors of  \cite{MW17,
CETW23} proved that their algorithms are expected
to output near-optimal {\em LRA} for Symmetric Positive
Semidefinite (SPSD) matrices.\footnote{The sublinear cost of the {\em LRA} algorithms of \cite{MW17,
BW18,CETW23}  does not include the
 superlinear cost of a posteriori  estimation of  their output error norms and correctness
verification, but the  deterministic algorithm of \cite[Part III]{LP20}, running at sublinear cost, computes {\em LRA} of an
$n\times n$ SPSD matrix
with both spectral and Frobenius error norms within a factor of $n$ from optimal and as by-product,  at no additional cost, estimates
error norm, verifying correctness.}

As we recalled already, sublinear cost C-A iterations  
 consistently 
output 
 accurate
{\em LRA} empirically for a large class of matrices. Based on  
advance analysis, 
some  limited formal support
has been obtained  
for such
empirical behavior 
of C-A 
iterations and/or their ingenious modification.

Namely, Chiu and Demanet in \cite{CD13}
prove
that a simple  algorithm
running at sublinear cost  outputs quite accurate
{\em LRA} in its memory efficient form of $CUR$ for $C$ and $R$ made up of $l$ and $k$ columns and rows of $M$, respectively,
provided that  a rank-$r$ approximation of $M$
admits SVD-like factorization
$XYZ^T$ where the column vectors of the $m\times r$ factor $X$ and $n\times r$ matrix $Z$ are
orthonormal
 and incoherent.
Chiu and Demanet also prove that this property for $X$ alone  is sufficient 
when they extend  
 their algorithm with a single C-A step. 

 Cortinovis and Ying in \cite{CY25} 
extend the
latter results
to the case where some orthogonal columns are sparse rather than incoherent
provided that an {\em LRA}  algorithm
generalizes C-A step  by using
``progressive
alternating direction pivoting"
of \cite[Alg.
 2.2]{LYMHY}, \cite{X24} -- it combines uniform
random choice of   $s<r$  columns and rows of $M$ for the factors
$C$ and $R$ of {\em CUR LRA}
with expanding and  updating these sets by means
of the Interpolative Decomposition of \cite[Sec. 3.2.3]{HMT11} based on Strong Rank Revealing  QR
factorization of \cite{GE96}.

The important works of \cite{CD13,CY25}
propose novel advanced techniques
for the study of primal {\em LRA},
while our  theorems are on dual {\em LRA} (for random inputs $M$); our semi-heuristic sublinear cost   modifications of the random sketching {\em LRA},
 motivated 
by these theorems,
are also very much
distinct from
the techniques of \cite{CD13,
LYMHY,X24,
CY25}.

 We extracted
these  sublinear cost modifications of random sketching {\em LRA} as well as  both Thms. \ref{eqtherr} and \ref{therrfctr} 
from our unpublished  report \cite{PLSZb}, 
which cites
\cite{PQY15, 
PZ17,
PLSZ16,
PLSZ17}
 as its predecessors and states motivation for further research on sublinear cost {\em LRA} as its major goal.
Encouraged by  appearance of \cite{X24,
CY25},
we hope that
publication of our results should also contribute to that goal.
\\  
{\bf 1.5. Organization of the paper.} 
%\label{sorgp} 
%------------------------------------------------------------------------------  
%At the end of this section and 
 Sec. \ref{ssdef}  is devoted to  background 
on matrix computations.  
In Secs. \ref{sbsalg} and \ref{sdetrerr}   we recall 
sketching algorithms for {\em LRA} and their deterministic output error bounds, respectively. In Sec. \ref{serrranin} we prove
 error bounds for our dual {\em LRA} algorithms --  our main results. 
  In Sec. \ref{srndsmpl} 
 we cover numerical tests. 
In Appendix \ref {snrmg} we recall the known estimates for the norms of a Gaussian matrix and its pseudo inverse.
In Appendix \ref{srndoprpr} we prove that pre-processing with Gaussian sketch matrices turns any matrix that admits  $LRA$ into a perturbed factor-Gaussian matrix.
 In Appendix \ref{slmmalg1} we recall
the error bounds for some known sketching algorithms.
  In Appendix \ref{shrdin} we  
    specify some small families of input matrices on which any sublinear cost {\em LRA}  fails. 
     In Appendix \ref{spreprmlt} 
 we  generate two  families of Ultrasparse sketch matrices. 
%In Appendix \ref{svlmfct} we estimate the volume and $r$-projective volume of a factor-Gaussian matrix, which are  basic parameters in the study of  {\em C-A} iterations.

\medskip
 
%------------------------------------------------------------------------------
 
\section{Background on matrix computations}\label{ssdef}

%------------------------------------------------------------------------------ 
%------------------------------------------------------------------------------
 
\subsection{Definitions and two lemmas}\label{sbckgm}  
  
%------------------------------------------------------------------------------
To simplify our presentation we assume  dealing
with real matrices throughout,
 except for  Appendix \ref{spreprmlt}, but our study can be quite
 readily extended to the case of complex matrices.

\begin{itemize} 
  \item%1 
$M^T$ denotes the transpose of a matrix
$M$.
 \item%3
 $||\cdot||_2$ and $||\cdot||_F$ denote the spectral and  Frobenius  matrix norms, respectively;  we write
$|||\cdot|||$ where a property holds under both  of these norms (cf. \cite[Thm. 9.1]{HMT11}).

\item%2 
  A {\em (compact) singular value decomposition (SVD)} of an $m\times n$ matrix $M$ of a rank $\rho$ (cf. \cite[page 31]{B15}) is the decomposition $M=U\Sigma  V^T$ where $\Sigma=
 \diag(\sigma_j)_{j=1}^{\rho}$ is
  the diagonal matrix   of the  singular values of $M$, $\sigma_1\ge \sigma_2\ge\dots\ge \sigma_\rho>0$, and  $U\in \mathbb R^{m\times \rho}$ and $V\in \mathbb R^{n\times \rho}$  are
  two  matrices with orthonormal columns, filled with the associated left and right singular spaces, 
  respectively.
 
\item%
 For $r\le\rho$ and a rank-$\rho$ matrix $M$
define its
 {\em rank-$r$ truncation}, $M_{r}$, obtained  by setting 
$\sigma_j(M)=0$ for $j>r$. Its (compact) SVD is said 
to be $r${\em -top SVD} of $M$.
It is a rank-$r$ approximation of $M$ having minimal spectral and Frobenius error norms
\begin{equation}\label{eqecky}
\tilde \sigma_{r+1}(M):=|||M-M_r|||,
\end{equation}
by virtue of Eckart--Young--Mirski theorem
(cf. \cite[page 79]{GL13}), where
$\tilde \sigma_{r+1}(M)= \sigma_{r+1}(M)$
under spectral norm and
$\tilde \sigma_{r+1}(M)^2=\sum_{j>r}\sigma_{j}(M)^2$ under Frobenius norm.

%------------------------------------------------------------------------------

\item%4
%QL0127
%V {\color{red} use of norm in the subscript?} Thank you, fixed!
 $\rank(M)$  denotes the  {\em rank} 
of a matrix $M$. 
$\epsilon$-$\rank(M)$ is argmin$_{|||E|||\le\epsilon|||M|||}\rank(M+E)$; it is called
 {\em numerical rank}, $\nrank(M)$,  if a tolerance 
  $\epsilon$ is small in context, typically 
being  linked to machine precision or the level of relative errors of the computations (see \cite[page 276]{GL13}). 
\item%3
$M^+$ denotes the Moore -- Penrose pseudo inverse of $M$.

 \item%8 
For a matrix $M=(m_{i,j})_{i,j=1}^{m,n}$ and two sets $\mathcal I\subseteq\{1,\dots,m\}$  
and $\mathcal J\subseteq\{1,\dots,n\}$,   define
the submatrices
%\begin{equation}\label{eq111}
$M_{\mathcal I,:}:=(m_{i,j})_{i\in \mathcal I; j=1,\dots, n},  
M_{:,\mathcal J}:=(m_{i,j})_{i=1,\dots, m;j\in \mathcal J},~{\rm and}~  
M_{\mathcal I,\mathcal J}:=(m_{i,j})_{i\in \mathcal I;j\in \mathcal J}.$ 
%\end{equation}
  
%QL0127 added definition of row/column space of a matrix
\item
%V For matrix $M = (m_{i,j})_{i,j=1}^{m,n}$, %V its row space is defined as 
$\textrm{Span}(M_{1,:}^T, M_{2,:}^T, ..., M_{m,:}^T )$ denotes the row space
%QL0221  "."  --> ":" 
of a matrix $M = (m_{i,j})_{i,j=1}^{m,n}=(M_{i,:}^T)_{i=1}^m=(M_{:,j})_{j=1}^n,$ 
and 
$\textrm{Span}(M_{:, 1}, M_{:, 2}, ..., M_{:, n})$ denotes its column space.
 \end{itemize}
 
%------------------------------------------------------------------------------
 
%\subsection{Auxiliary results}  
   
%------------------------------------------------------------------------------
 
 \begin{lemma}\label{lehg} {\rm [The norm of the pseudo inverse of a matrix product  (cf. \cite{B15}).]} Suppose that $A\in\mathbb R^{k\times r}$, $B\in\mathbb R^{r\times l}$, and the matrices $A$ and   $B$ have full rank $r\le \min\{k,l\}$. Then
 $$|||(AB)^+|||\le |||A^+|||~~|||B^+|||.$$ \end{lemma}

%------------------------------------------------------------------------------
%------------------------------------------------------------------------------

% \begin{lemma}\label{lesngr} {\rm [The impact of a perturbation of a matrix on its singular values (see \cite[Cor. 8.6.2]{GL13}).]} For $m\ge n$ and a pair of ${m\times n}$ matrices $M$ and $M+E$ it holds that $$|\sigma_j(M+E)-\sigma_j(M)|\le||E||_2~{\rm for}~j=1,\dots,n. $$\end{lemma} 

%------------------------------------------------------------------------------

%QL0126

\begin{lemma}\label{lemma:pert_sing_space} {\rm (The impact of a   perturbation of a matrix on its singular space, adapted from \cite{W72},
 \cite[Thm. 6.4]{S73}, \cite[Thm. 1]{GT16}.)}
Let $M$ be an $m\times n$ matrix of rank $r < \min (m, n)$ where
\begin{eqnarray*}
M = 
\begin{bmatrix}
U_r & U_{\perp}
\end{bmatrix}
~
\begin{bmatrix}
\Sigma_r & 0 \\
0 & 0
\end{bmatrix}
~
\begin{bmatrix}
V_r^T \\
V_{\perp}^T
\end{bmatrix}
\end{eqnarray*}
is its SVD, and let $E$ be a perturbation matrix  such that 
\begin{eqnarray*}
\delta = \sigma_r(M) - 2~||E||_2 > 0
%\end{equation}
~{\rm and}~ 
%Vthat
%\begin{equation}
||E||_F \le \frac{\delta}{2}.
\end{eqnarray*}
Then there exists a matrix  such that  $P\in\mathbb R^{(n-r)\times r}$,  $||P||_F < 2~\frac{||E||_F}{\delta} < 1$, and the columns of the matrix
$\tilde{V} = V_r + V_{\perp}P$ span the right leading singular subspace 
of $\tilde M = M + E$.  
\end{lemma}

\begin{remark}\label{remark:pert_sing_space}
Matrix $\tilde{V}$ from the above  does not necessarily have orthonormal  columns, but the matrix 
$(V_r + V_{\perp}P)(I_r + P^TP)^{-1/2}$
has orthonormal columns, {\em i.e.}, $(I_r + P^TP)^{-1/2}$ normalizes $\tilde V$ (see  \cite{W72,S73,
GT16}).
\end{remark}

%------------------------------------------------------------------------------

\subsection{Gaussian  and factor-Gaussian 
%%Vrandom
matrices}

%------------------------------------------------------------------------------
%------------------------------------------------------------------------------

%\noindent  Our estimates involve  the %norms of  a Gaussian matrix and its %pseudo inverse  (cf. Section%\ref{snrmg}). 
%Hereafter we let $\eqid$ denote equality in distribution.
{\em Constant matrices} are filled with constants, unlike random matrices, filled with random variables.

%We let $\mathcal G^{p\times q}$ denote a $p\times q$ Gaussian matrix, and define random variables
%$\nu_{p,q} \eqid |G|$, $\nu_{{\rm sp},p,q} \eqid ||G||$, $\nu_{F,p,q} \eqid ||G||_F$,
%$\nu_{p,q}^+ \eqid |G^+|$,  $\nu_{{\rm sp},p,q}^+ \eqid ||G^+||$, and
%$\nu_{F,p,q}^+ \eqid ||G^+||_F$, for  a $p\times q$ random Gaussian matrix $G$.
%[$\nu_{p,q} \eqid \nu_{q,p}$  and
%$\nu_{p,q}^+ \eqid \nu_{q,p}^+$, 
%for all pairs of $p$ and $q$.] 

%QL0221 
%V?Hereafter we consider {\it constant matrices} (as opposed to random matrices) to be matrices all of whose entries are 
%V?constants (rather than random variables), and we 
%V? would specify which matrices are constant whenever confusion may arise.   

\begin{theorem}\label{thrnd} {\em [Non-degeneration of a Gaussian  
%%Vrandom 
matrix.]}
Suppose that 
%$F \eqid \mathcal G^{r\times p}$, 
%$H \eqid \mathcal G^{q\times r}$, 
%$M\in  \mathbb R^{p\times q}$ 
$M\in\mathbb R^{p\times q}$ is a constant matrix, $r\le\rank(M)$,
and $F$ and $H$ are $r\times p$ and $q\times r$ independent  Gaussian 
%%V random 
matrices, respectively.
Then  
the matrices $F$, $H$, $FM$,  and $MH$  
have full rank $r$ 
with probability 1.
\end{theorem} 
%QL0221 added proof of non-degeneration of a Gaussian random matrix 
\begin{proof}
 Rank deficiency of matrices 
 $H$, $FM$, and $MH$ 
 %V0221 are 
is equivalent to turning into 0
the determinants
$\det(FF^T)$, $\det(H^TH)$, $\det((MH^T)MH)$, and $\det(FM(FM)^T)$,
respectively. 
The claim follows because 
%V0221 this equation defines  an 
 these equations define algebraic varieties of  lower
dimension in the linear spaces of the entries, considered independent variables (cf., e.g., \cite[Prop. 1]{BV88}).
\end{proof}

%QL0221
 \begin{remark}\label{refllrnk}
Events that  occur with probability 0 are  immaterial
 for  our probability estimates, and
 hereafter we say that a  matrix  has full rank even  if it is rank deficient with probability 0.  \end{remark}
 
% {\color{red}
%  \begin{remark}
% Hereafter we condition on the event that such matrices do have full rank whenever their rank deficiency  
% with probability 0 is immaterial
% for deducing  our probability estimates.
% \end{remark} }
%------------------------------------------------------------------------------
%------------------------------------------------------------------------------

%\begin{assumption}\label{assmp1} We simplify the statements of our results by assuming that a Gaussian matrix  has full rank and  ignoring the probability 0 of its degeneration. 
%\end{assumption}

%------------------------------------------------------------------------------

%------------------------------------------------------------------------------

\begin{lemma}\label{lepr3} {\rm  [Orthogonal Invariance.] \cite[Theorem 3.2.1]{T12}}.
Suppose that  
$G$ is an $m\times n$ Gaussian 
%%V random  
matrix, 
$k\le \min\{m,n\}$ is a positive integer, and $S\in\mathbb R^{k\times m}$ and  $T\in\mathbb R^{n\times k}$ are constant matrices,  
 having orthonormal rows and columns, respectively.
Then $SG$ and $GT$ are random matrices having  distribution of  $k\times n$ and $m\times k$ Gaussian  random matrices, respectively.
\end{lemma}

%-----------------------------------------------------
%------------------------------------------------------------------------------

\begin{definition}\label{deffctrg}  {\em [Factor-Gaussian matrices.]} 
Let 
$A\in \mathbb R^{m\times r}$, 
$B\in \mathbb R^{r\times n}$, and
$C\in \mathbb R^{r\times r}$ be 
three constant well-conditioned matrices of full rank $r<\min\{m,n\}$.
Let $G_1$ and $G_2$ be $m\times r$ and $r\times n$ independent Gaussian 
%%V random 
matrices, respectively.
Then  we call the  matrices 
$G_1B$, $AG_2$, and 
$G_1 C G_2$
 {\em left},
 {\em right}, and  
 {\em two-sided factor-Gaussian  
 %%V random 
 matrices
 of  rank} $r$, respectively.
\end{definition} 
 
 \begin{theorem}\label{thfctrg}
Any  two-sided  factor-Gaussian matrix 
%$G_{m,r} \Sigma G_{r,n}$
$G_{1} C G_{2}$ 
shares  probability distribution with the matrix 
$G_{1} \Sigma_C G_{2}$ for some diagonal matrix 
$\Sigma_C=
(\sigma_j)_{j=1}^r$ such that  
$\sigma_1\ge \sigma_2\ge \dots\ge \sigma_r>0$.
\end{theorem}
\begin{proof}
Let $C=U_C\Sigma_C V_C^*$ be SVD.
%QL0418
Then the matrices $G_{1}U_C$ 
and $V_C^*G_{2}$
have distributions of  $m\times r$ and $r\times n$ Gaussian matrices, respectively, 
by virtue of Lemma \ref{lepr3}. 
\end{proof} 
%%------------------------------------------------------------------------------
%
%\begin{definition}\label{defrelnrm}  
%{\rm The relative norm of a perturbation of  a
%  Gaussian matrix} is the ratio of the perturbation norm and the expected value of the norm of the matrix (estimated in Theorem \ref{thsignorm}).  
%\end{definition}
%   
%%------------------------------------------------------------------------------
%  
%We refer to all three matrix classes above as {\em factor-Gaussian matrices
% of  rank} $r$, to their perturbations within
% a relative norm bound $\epsilon$ as {\em factor-Gaussian matrices
% of $\epsilon$-rank} $r$, and to their 
%  perturbations within a small relative norm  as  {\em factor-Gaussian matrices
% of numerical rank} $r$, to which we also refer as {\em perturbations of factor-Gaussian matrices}. 
% 
% Clearly $||(A\Sigma)^+||\le ||\Sigma^{-1}||~||A^+||$ and $||(\Sigma B)^+||\le ||\Sigma^{-1}||~||B^+||$ for
%a two-sided factor-Gaussian matrix 
%$M=A\Sigma B$ of rank $r$ of Definition \ref{deffctrg}, and so
%whp such a matrix is both left and right 
%factor-Gaussian of rank $r$.

%------------------------------------------------------------------------------

\section{{\em LRA} by means of sketching: four  algorithms}\label{sbsalg}
 
%------------------------------------------------------------------------------

Next we slightly generalize the sketching  {\em LRA} algorithms of \cite{HMT11,
TYUC17}
(see Remarks \ref{realg10}
 -- \ref{resmpl}).

%------------------------------------------------------------------------------

\begin{algorithm}\label{alg1} %{\rm LRA via Randomized  Range Finder.} 
   
%------------------------------------------------------------------------------
 
\begin{description}

%------------------------------------------------------------------------------

\item[{\sc Input:}] 
An $m\times n$ matrix  $M$ and a
 target  rank  $r$.  

%------------------------------------------------------------------------------

\item[{\sc Output:}] 
Two matrices $X\in \mathbb R^{m\times l}$ and 
$Y\in \mathbb R^{l\times n}$ for $r\le l\le n$ defining
an {\em LRA} $\tilde M=XY$ of $M$.
%and the relative error $||\tilde M-M||/||M||$.    

%------------------------------------------------------------------------------

\item[{\sc Initialization:}] 
 Fix an integer $l=r+p\le n$, for $p\ge 0$, and
 an $n\times l$  matrix $H$ of full rank $l$.

%------------------------------------------------------------------------------ 
 
\item[{\sc Computations:}]
%\item (cf. items 1.4.6 and 1.4.10): $~$

\begin{enumerate}
\item %1
Compute the  $m\times l$ matrix $MH$.
% by using $(2n-1)ml$ flops. 
\item %2
Fix a nonsingular $l\times l$ matrix $T^{-1}$  and output  the $m\times l$ matrix $X:=MHT^{-1}$.
\item %3  
Output an $l\times n$ matrix 
%\begin{equation}\label{eqv}
$Y:= {\rm argmin}_V ~|||XV-M|||=X^+M$. 
%\end{equation} 
\end{enumerate}

%------------------------------------------------------------------------------

\end{description}

%------------------------------------------------------------------------------
 
\end{algorithm}

%----------------------------------------------------------------------------

\begin{algorithm}\label{alg0} %{\rm LRA via  Transposed Randomized Range Finder.}
%( See Remark \ref{realg0}.)} 
   
%------------------------------------------------------------------------------
 
\begin{description}

%------------------------------------------------------------------------------

\item[{\sc Input:}] 
As in Alg. \ref{alg1}.

%An $m\times n$ matrix  $M$ and an integer
  %$r$ (the target  rank).  

%------------------------------------------------------------------------------

\item[{\sc Output:}] 
Two matrices $X\in \mathbb R^{k\times n}$ and 
$Y\in \mathbb R^{m\times k}$ defining
an {\em LRA} $\tilde M=YX$ of $M$.

%------------------------------------------------------------------------------

\item[{\sc Initialization:}] 
 Fix an integer $k=r+p\le m$, 
 for $p\ge 0,$ and
 a $k\times m$  matrix $F$ of full numerical rank $k$.

%------------------------------------------------------------------------------ 
 
\item[{\sc Computations:}]
%\item (cf. items 1.4.6 and 1.4.10): $~$

\begin{enumerate}
\item %1
Compute the  $k\times m$ matrix $FM$.
% by using $(2m-1)nk$ flops.
\item %2
Fix a nonsingular $k\times k$ matrix $S^{-1}$; then output
 $k\times n$ matrix $X:=S^{-1}FM$.
\item %3  
Output an  $m\times k$ matrix 
%\begin{equation}\label{eqv}
$Y:= {\rm argmin}_V ~|||VX-M|||=MX^+$.
%[$Y=M(FM)^+$ if the matrix $FM$ has full rank $k$.]
%\end{equation} \rank(FM)=k
\end{enumerate}

%------------------------------------------------------------------------------

\end{description}

%------------------------------------------------------------------------------
 
\end{algorithm}

%------------------------------------------------------------------------------

\noindent 
%\begin{figure}[htb] 
%\centering
%\includegraphics[scale=0.25] {"PSZf_06".png}
%\caption{The matrices of Algorithm \ref{alg1}.}
%\label{fig01}
%\end{figure}

%------------------------------------------------------------------------------

The following
algorithm combines row and column sketching. %For $T$ being a square factor $R$ of QR factorization of $MH$  and for $S$ being the $k\times k$ identity matrix $I_k$, it turns into the algorithm of  \cite[Sec. 1.4]{TYUC17}, whose  origin is traced back to \cite{WLRT08}. 
%\footnote{The algorithms and estimates of \cite[Theorems 4.7 and 4.8]{CW09} use Rademacher (rather than Gaussian) %matrices.} 
  
\begin{algorithm}\label{alg01} %{\rm Generalized   Nystr{\"o}m algorithm.} 

%------------------------------------------------------------------------------
 
\begin{description}

%------------------------------------------------------------------------------

\item[{\sc Input:}] 
As in Alg. \ref{alg1}.

%------------------------------------------------------------------------------
%------------------------------------------------------------------------------

\item[{\sc Output:}] 
Two matrices $X\in \mathbb R^{m\times k}$ and 
$Y\in \mathbb R^{k\times n}$ defining
an {\em LRA} $\tilde M=XY$ of $M$.

%------------------------------------------------------------------------------

\item[{\sc Initialization:}] 
 Fix two
 integers $k$ and $l$, $r\le k\le m$
 and $r\le l\le n$; fix
 two  matrices $F\in\mathbb R^{k\times m}$ 
and $H\in\mathbb R^{n\times l}$  
 of full numerical ranks
 and two nonsingular matrices $S\in\mathbb R^{k\times k}$ and $T\in\mathbb R^{l\times l}$.

%------------------------------------------------------------------------------ 
 
\item[{\sc Computations:}]
1. Output the matrix  
$X=MHT^{-1 }\in\mathbb R^{m\times l}$.

 2. Compute  the matrices  
 $U:=S^{-1}FM\in\mathbb R^{k\times n}$ and
$W:=S^{-1}FX\in\mathbb R^{m\times l}$.

3. Output the $l\times n$ matrix 
%\begin{equation}\label{eqv1}
$Y:={\rm argmin}_V |||WV-U|||=W^+U$.
%\end{equation}
%\end{enumerate}

%------------------------------------------------------------------------------

\end{description}

%------------------------------------------------------------------------------
 
\end{algorithm}
   
%------------------------------------------------------------------------------

%------------------------------------------------------------------------------

\begin{remark}\label{realg10} 
 We can obtain Algs. \ref{alg0}
  by applying Alg. \ref{alg1}
  to the transpose $M^T$.
Likewise,
by applying
Alg. \ref{alg01}
 to $M^T$ we obtain
 {\bf Alg.  3.4}.
We only study  Algs. \ref{alg1} and \ref{alg01}, but can  readily extend that study to 
Algs. \ref{alg0}
and 3.4.  \end{remark}

\begin{remark}\label{re0} 
Fix $k=r$,  $l>k$, a random  $n\times l$
matrix $H$ (e.g.,  Gaussian, SRHT, or SRFT matrix), identity matrix $S$, and $T$ equal to the $R$ factor in the $QR$ factorization of $MH$. Then the matrix $X$ has  orthonormal
columns and 
Algs. \ref{alg1}
turns into Proto-algorithm of \cite[Sec. 9]{HMT11}, while Alg.
 \ref{alg01}
turns into the {\rm Generalized   
Nystr{\"o}m algorithm} 
of \cite[Eqn. (3)]{N20}.
\cite{N20} stabilizes the latter  algorithm numerically
 -- essentially by means of setting to 0 all
singular
values of the matrix $W$ exceeded by a fixed 
 $\epsilon$,  ``a modest multiple
of the unit roundoff $u$ times $||W||$". 
Our  study can be readily extended  to such a stabilized
{\em LRA} because the stabilization 
little affects  complexity of {\em LRA} and  
only improves its output  accuracy.
We can obtain
various other modifications 
of Algs. \ref{alg1} and \ref{alg01}
 by fixing other
sketching matrices $F$ and $H$.

 \end{remark}

%------------------------------------------------------------------------------ 
%------------------------------------------------------------------------------

\begin{remark}\label{resmpl}Column (resp. row) sketching turns into  column (resp. row) {\em subset selection} where $H$ in Alg. \ref{alg1}  (resp. $F$ in Alg. \ref{alg0}) is a {\em sampling matrix}, that is, a  full rank submatrix of a permutation matrix.  \end{remark}

%------------------------------------------------------------------------------
         
%\begin{figure}[htb]
%\centering
%\includegraphics[scale=0.25] {"CURFig6".png}
%\caption{The matrices of Algorithms \ref{alg1} and %\ref{alg01} (the latter ones are shown by black).}%\label{Fig0101}
%\end{figure}
   
%------------------------------------------------------------------------------

%------------------------------------------------------------------------------

%\begin{remark}\label{realg1} Let the matrices $MH$, $FM$, and $FMH$ of  Algs.  \ref{alg1},\ref{alg0}, and  \ref{alg01},   have full ranks $l$, $k$, and $\min\{k,l\}$,respectively.Then$XY=MH(MH)^+M$,$YX=M(FM)^+FM$, and $YX=MH(FMH)^+FM$, respectively,independently of the choice of non-singular matrices$S$and$T$The choice of the matrices  $S$ and $T$ affects the conditioningofthe matrices$X=MHT^{-1}$,$X=S^{-1FM$ and $W=S^{-1}FMHT^{-1}$. In particular   $X$ has orthonormal columns in  Alg. \ref{alg1} if $T$ is the factor $R$ in the thin QR factorizationof$MH$as well as if $T=R\Pi$ and if $R$ and $\Pi$ are factors of arank-revealing {\rm QR\Pi$ factorization of $MH$.\end{remark}

%------------------------------------------------------------------------------
   
\section{Deterministic output error bounds for  sketching algorithms}\label{sdetrerr}
  
%--------------------------------------------------------
%--------------------------------------------------------
  
\subsection{Deterministic error bounds of Range Finder}\label{serrrng}

%------------------------------------------------------------------------------
                                                                                                                                                                                                                                                                                                                                                                                  Next we recall some known estimates for the errors of Alg. \ref{alg1}, to be used in the next section.
                                                                                                                                                                                                                                                                                                                                                                                  
                                                                                                                                                                                                                                                                                                                                                                                  \begin{theorem}\label{thpert1} {\rm \cite[Thm. 9.1]{HMT11}.}
Suppose that Alg. \ref{alg1}
has been applied to a matrix $M$ and let  
%QL0127  
%V {\color{red} (svd for $M$ seems wrong? It seems Corollary 4.1 assumes $H$ is orthogonal.)}

\begin{eqnarray*}
%V\label{eqmmr} 
 M=\begin{pmatrix} U_1& U_2\end{pmatrix}
\begin{pmatrix}\Sigma_1& \\
 &\Sigma_2 \end{pmatrix}\begin{pmatrix}~V_1^*\\~V_2^*\end{pmatrix}~~{\rm and}~
 M_r=U_1\Sigma_1 V_1^*
\end{eqnarray*} 
  be SVDs
 of the matrices $M$ and its rank-$r$ truncation $M_r$, 
 respectively. [$\Sigma_2=O$ and $XY=M$ if 
 $\rank(M)=r$. The  $r$ columns of $V_1$ are the $r$ top right singular vectors of $M$.]
 Write 
   \begin{equation}\label{eqc12}
C_1=V^*_1H,~  
 C_2=V^*_2H.
 \end{equation}
 %Then 
 %QL0221 add in assumption for H and that C_1 has full rank
Assume that $||H||_2 \le 1$ and $\rank(C_1) = r$. Then
%V0221 , then
\begin{equation}\label{eqerrnrm} 
 |||M-XY|||^2\le
 %QL0127 added missing "^2" 
 |||\Sigma_2|||^2+|||\Sigma_2C_2C_1^+|||^2.
 \end{equation}
 \end{theorem}

%QL0221 
 
\begin{corollary}\label{copert0}
Under the assumptions of Thm. \ref{thpert1} and for
$\tilde \sigma_{r+1}(M)$ of (\ref{eqecky}) it holds that
\begin{equation}\label{eqmmr1}
|||M-XY|||/\tilde \sigma_{r+1}(M) \le (1+|||C_1^+|||^2)^{1/2} ~{\rm for}~C_1=V_1^*H.
\end{equation}
%V0221

%Under the assumptions of Theorem \ref{thpert1} it holds that
%%QL0127 
%% {\color{red} We defined $\tilde \sigma$ in definition but are using $\bar \sigma$ here}
%\begin{equation}\label{eqmmr1}
%|||M-XY|||\le (1+|||C_1^+|||^2)^{1/2}\tilde \sigma_{r+1}(M)~{\rm for}~C_1=V_1^*H
%\end{equation}
%for $\tilde \sigma_{r+1}(M)$ of (\ref{eqecky}).%QL0127.
\end{corollary}

\begin{proof}
The corollary follows from (\ref{eqerrnrm}) because
$$|||\Sigma_2|||=\tilde\sigma_{r+1}(M),~ |||C_2|||\le 1,~
{\rm and}~|||\Sigma_2C_2C_1^+|||\le |||\Sigma_2|||~~|||C_2|||~~|||C_1^+|||.$$  
%The corollary follows from (\ref{eqerrnrm}) because
%$$|||\Sigma_2|||=\bar\sigma_{r+1}(M),~ |||C_2|||\le 1,~
%{\rm and}~|||\Sigma_2C_2C_1^+|||\le |||\Sigma_2|||~~|||C_2|||~~|||C_1^+|||.$$  
\end{proof}

%V0221
(\ref{eqmmr1}) implies
that the output {\em LRA} is optimal under both spectral and Frobenius matrix norms up to a factor of $(1+|||C_1^+|||^2)^{1/2}$.

%QL0221
%{\color{red}
%\begin{corollary}%\label{copert0}
%Under the assumptions of Theorem \ref{thpert1} it holds that
%\begin{equation}%\label{eqmmr1}
%|||M-XY|||/\tilde \sigma_{r+1}(M) \le (1+||C_1^+||^2)^{1/2} ~{\rm for}~C_1=V_1^*H
%\end{equation}
%for $\tilde \sigma_{r+1}(M)$ of (\ref{eqecky}).
%%\footnote{ Recall that ratio $|||M-XY|||/\tilde \sigma_{r+1}(M)$
%%measures the approximation error against optimal error in Frobenius and/or Spectral norm,
%%that is in Frobenius norm $|||M-XY|||/\tilde \sigma_{r+1}(M) := ||M-XY||_F / \sqrt{\sum_{i>r}\sigma_i^2(M)} $,
%%and in Spectral norm $|||M-XY|||/\tilde \sigma_{r+1}(M) := ||M-XY||_2/\sigma_{r+1}(M)$.}
%\end{corollary}
%\begin{proof}
%The corollary follows from (\ref{eqerrnrm}) because
%$$|||\Sigma_2|||=\tilde\sigma_{r+1}(M),~ |||C_2|||\le 1,~
%{\rm and}~|||\Sigma_2C_2C_1^+|||\le |||\Sigma_2|||~~||C_2||~~||C_1^+||.$$  
%\end{proof}
%}  

%------------------------------------------------------------------------------
    
\subsection{Impact of pre-multiplication on the errors of  {\em LRA} }\label{simppre}

%------------------------------------------------------------------------------
 
 The following  theorem shows that 
the overall  error bounds of Alg. \ref{alg01}  
are dominated by the product of the norm $|||W|||$ and the
error norm bound  of Alg. \ref{alg1}.

\begin{theorem}\label{thst3}  {\rm See \cite{TYUC17}.} 
Let Alg. \ref{alg01} output  a matrix $XY$ for 
$Y=(FX)^+FM$ and let $m\ge k\ge l=\rank(X)$. Then
\begin{equation}\label{equvmpre}
 M-XY=W(M-XX^+M)~{\rm for}~W=I_m-X(FX)^+F,
\end{equation}
%and hence
\begin{equation}\label{eqf} 
 |||M-XY|||\le |||W|||~~|||M-XX^+M|||,~~|||W|||\le  
 |||I_m|||+|||X|||~~|||F|||~~|||(FX)^+|||.
\end{equation} 
 \end{theorem}
 \begin{proof} 
 Recall that $Y=(FX)^+FM$ and notice that 
 $(FX)^+FX=I_l$
 if $k\ge l=\rank(FX)$. 
 Therefore,
 $Y=X^+M+(FX)^+F(M-XX^+M)$. Consequently, (\ref{equvmpre}) and  (\ref{eqf}) hold.
\end{proof} 
 
\begin{remark}\label{reprm}
  Deduce that
$|||W|||\le
|||I_m|||+|||F|||~|||F^+|||~|||X|||~|||X^+|||$ by combining bound (\ref{eqf})
with Lemma \ref{lehg}.
Hence $$|||W|||\le |||I_m|||+1$$ 
 if  both matrices $F$  and $X$ have orthonormal columns. The sketch matrix  $F$ is our choice, and for $X$ we ensure column  orthogonality   by properly choosing the matrix $T$ in Alg. \ref{alg01} 
 (see Remark \ref{re0}).
\end{remark}

%\noindent 

%Next we   complement the bounds on the norm $||M-XX^+M||_2$ of the previous subsection by the bounds on the norms $||(FX)^+||_2$ and $||W||_2$.  

%For any fixed constant $h>1$ the algorithms of \cite{GE96,P00} output sampling matrices $F$  of size $k\times m$ satisfying \begin{equation}\label{eq||FX||}|| (FX)^+||_2\le ||X^+||_2\sqrt{(m-k)kh^2+1}.\end{equation}For $W=I_m+X(FX)^+F$ of (\ref{equvmpre}),it follows that $||W||_2\le 1+||X^+||_2\sqrt{(m-k)kh^2+1}\approx ||X^+||_2\sqrt {mk+1}$ for $m\gg k$ and  $h\approx 1$. The latter upper bound is pessimisticaccording to our numerical  tests, in which  output errors tend to decrease, and frequently significantly, when we double or triple $k$. In particular \cite[Alg. 1]{P00}  computes such a matrix $F$ in $O(mk^2)$ flops, that is, at sublinear cost for$k^2\ll  n$.
 
%------------------------------------------------------------------------------   
%-----------------------------------------------------------
%
%section accuracy of sublinear cost dual lra
%\input{section4.tex}

\section{Output error norm bounds for dual  sketching algorithms}\label{serrranin}
 
%------------------------------------------------------------------------------

Given the matrices 
$MHT^{-1}$ and $S^{-1}FM$,  
Alg. \ref{alg01} uses  
$O(kln)$ flops and hence
runs at sublinear cost where $kl\ll m$.
If also $l^2\ll m$ and $k^2\ll n$, then
for proper   Ultrasparse  matrices $F$ and $H$  we can  compute the matrices  $MHT^{-1}$ and $S^{-1}FM$ at sublinear cost as well
and hence can perform entire Algorithm  \ref{alg01} at sublinear cost, involving much less than $mn$ entries of $M$ and other scalars  already where $k\ll m$ and $l\ll n$.  Although such an {\em LRA}  fails for a worst case input (cf. Appendix \ref{shrdin}),  it succeeds
whp in the case of any fixed, possibly sparse, well-conditioned  matrix $H$ of full rank
%%V 
and a random input matrix $M$ that admits {\em LRA}, as we prove next. In view of the previous subsection we only need to prove this for Alg. \ref{alg1}.
  
% By applying the results of the previous section we extend these upper estimates for output accuracy to  variations of Algorithms \ref{alg01} and  3.4
% that run at  sublinear cost.

\subsection{Auxiliary results}

%V? We first deduce a lower bound on the smallest singular value of a matrix product $QH$, where the rows of $Q$  form an orthonormal basis of a randomly generated linear subspace and where $H$ is any fixed full rank matrix,  possibly sparse.  

\begin{lemma}\label{lemma:least_sing_bound}
Suppose that $G$ and $H$ are $r\times n$ and $n\times l$ matrices, respectively,  $r < l < n$, $GH$ has full rank $r$, and  $Q$ is an $r\times n$ matrix with orthonormal rows such that $Q$ and $G$ have  the same row space. Then 
\begin{eqnarray*}
\sigma_r(QH) \ge \frac{\sigma_r(GH)}{\sigma_1(G)}.
\end{eqnarray*}
\end{lemma}

\begin{proof}
Without loss of generality, suppose that $R\in \mathbb  R^{r\times r}$ and $G = RQ$. Then
\begin{eqnarray*}
\sigma_r(QH) = \sigma_r(R^{-1}GH) \ge \sigma_r(GH)\cdot \sigma_r(R^{-1}).
\end{eqnarray*}
Hence
 $\sigma_r(R^{-1}) = 1/\sigma_1(G)$
 because the matrices $R$ and $G$ share their singular values.
\end{proof}

%V \noindent 
Hereafter  write  $e:=2.7182818\dots$.
\begin{lemma}\label{lemma:gauss_least_bound} 
Suppose that $G$ is an $r\times n$  Gaussian  
%%V random
matrix, $H$ is an  $n\times l$ matrix with orthonormal columns,
 $n > l > r$, $l \ge 4$, 
 $Q$ is a matrix with orthonormal rows, and  $Q$ and $G$ share their row space. 
%QL0221 let t_2 < 1  
Fix two positive parameters $t_1$ and $t_2<1$. 
Then  
\begin{eqnarray*}\label{eqpr}
\sigma_r(QH) \ge  t_2 \cdot \frac{\sqrt{l}- \sqrt{r/l}+ \sqrt{1/l}}{t_1 + \sqrt{r} + \sqrt{n} }\cdot \frac{1}{e} \
\end{eqnarray*}
with a probability no less than $1 - \exp (-t_1^2/2) - (t_2)^{l-r}$.
\end{lemma}

\begin{proof}
The matrix  $GH$ has  distribution of an $r\times l$ Gaussian 
   random 
 matrix by virtue of 
 Lemma \ref{lepr3},
 and hence we can  assume that it has full rank (see Remark
\ref{refllrnk}).
%Notice that for a random Gaussian matrix of size $N\times n$ and $N > n$, it is very unlikely that its largest singular value is much greater than $\sqrt{N} + \sqrt{n}$, or its least singular value is much less than $\sqrt{N} - \sqrt{n}$. 
Now recall Thm. \ref{thsignorm} and claim (ii) of
Thm. \ref{thsiguna} and obtain    
\begin{eqnarray*}
\prob \{ \sigma_1(G) > t_1 + \sqrt{r} + \sqrt{n}  \} < \exp(-t_1^2/2)~{\rm for}~t_1 \ge 0~{\rm and}
%V\label{ineq:gauss_tail1}
\end{eqnarray*}
\begin{eqnarray*}
\prob \Big\{ \sigma_r(GH) \le t_2\cdot\frac{l - r + 1}{e\sqrt{l}}  \Big\} \le (t_2)^{l - r}~{\rm for}~t_2 < 1~{\rm and}~l \ge 4. 
%V\label{ineq:gauss_tail2}
\end{eqnarray*}
Combine the latter  two inequalities, 
%V?by applying 
the union  
%QL0221 footnote for union bound
bound,
%V?\footnote{Union Bound: $\prob \big\{ \cup_{i=1}^n A_i\big\} \le \sum_{i=1}^n (\prob \big\{ A_i\big\}) $ 
 %V? for events $A_1, ..., A_n$.}
and Lemma \ref{lemma:least_sing_bound}, and obtain  Lemma \ref{lemma:gauss_least_bound}.                                                                                                                                                                                                                                                                                                                                                                                                                                                                                                                                                                                                                                                                                                                                                                                                                                                                                                                                                                                                                                                                                              
\end{proof}

Lemma \ref{lemma:gauss_least_bound} implies that   $\sigma_r(QH)$ has at least order of $\sqrt{l/n}$ whp. 
%V?Next we specify this estimate under some reasonable bounds on $n$ and $l$. 

\begin{corollary}\label{coro:least_sing_bound}
For $n, l, r, G, Q$, and  $H$ of Lemma \ref{lemma:gauss_least_bound},  let 
$n >  36r$  and  $ l > 22(r - 1)$. Then 
\begin{eqnarray*}
\prob \big\{  \sigma_r(QH) \ge \frac{1}{4} \sqrt{l/n}  \big\} > 1 - \exp\Big(-\frac{n}{72}\Big) - \exp\Big(-\frac{l-r}{20}\Big).
%(0.95)^{l - r}.
\end{eqnarray*} 
\end{corollary}
\begin{proof}  
Write $t_1: = \frac{1}{3}\sqrt{n} - \sqrt{r}$ and $t_2: = \frac{1}{3}~\frac{le}{l-r+1}$, recall that 
$n >  36r$  and  $ l > 22(r - 1)$, 
and then readily verify that $t_1 > \frac{\sqrt{n}}{6}$ and $\exp(0.05) > 0.95 > t_2 > 0$. Finally apply Lemma \ref{lemma:gauss_least_bound} under these bounds on $t_1$ and $t_2$.
\end{proof}

\begin{remark}\label{remark:least_sing_bound} 
We   
can extend
 our lower bounds of  Lemma \ref{lemma:gauss_least_bound} and Cor. \ref{coro:least_sing_bound} on 
$\sigma_r(QH)$   to the case 
of any matrix $H$  of full rank $l$ if we decrease these bounds
by a factor of $\sigma_l(H)$. 
%V?rather than  an orthogonal matrix 
%V?if we decrease these bounds . 
\end{remark}

% and let $t_1 = \sqrt{n} - \sqrt{r}$, we have that
%\begin{equation}
%\prob \{ ||G||_2 > 2\sqrt{n} \} < \exp ()
%\end{equation}
%and let $t = \frac{6(l-r+1)}{el} > 1$, and further assume that $l > 2r -2$, we have that
%\begin{equation}
%\prob \{ \sigma_r (GH) \le \frac{\sqrt{l}}{6} \} \le t^{l - r} < (\frac{e}{3})^{l - r}.
%\end{equation}

%------------------------------------------------------------------------------
  
\subsection{Output errors of Alg. \ref{alg1}
for a matrix with random singular space}\label{sranssp}
 
%-------------------------------------------------} 
%-------------------------------------------------

\begin{assumption}\label{assump:mat_rand_space}
Let $r < n \le m$
(we can readily extend our study to the case where  $m < n$). Fix two constant matrices 
$$
%U\in \mathbb R^{m\times m}, 
\Sigma_r = 
%V \textrm
\diag(\sigma_j)_{j=1}^r~{\rm and }~
\Sigma_{\perp} = 
\begin{bmatrix}
%V \textrm{Diag}(\sigma_{r+1},\dots, 
%V\sigma_{n})
\diag(\sigma_j)_{j=r+1}^n\\
%V{\bf 0}
O_{m-n,n-r}
\end{bmatrix}\in \mathbb R^{(m-r)\times (n-r)}
$$
 %$U$ is orthogonal, and 
such that $\sigma_1\ge \sigma_2\ge \dots \ge \sigma_n \ge 0$, and $\sigma_r > 0$.
Let $G$ be an $r\times n$ Gaussian 
%%V random 
matrix 
and 
let $Q\in\mathbb R^{n\times r}$ and $Q_{\perp}\in\mathbb R^{(n-r)\times r}$ be  two matrices whose column sets 
make up orthonormal bases of the row space of $G$ and its orthogonal complement, respectively.
Furthermore, let 
\begin{equation}\label{eqassmp}
M = U\cdot\begin{bmatrix}\Sigma_r & 0\\ 0 & \Sigma_\perp \end{bmatrix} \cdot\begin{bmatrix}Q^T\\Q_\perp^T\end{bmatrix}
\end{equation}
be SVD where the
right
 singular space of $M$ is  random
 and $U$ is a  matrix
 with orthonormal columns.
\end{assumption}

\begin{remark}
% just by defining $\Sigma_\perp =\diag(\sigma_{r+1},\dots, \sigma_{m},0,\dots,0)$ such that Thm. \ref{eqtherr}  below still holds.
 The matrix $G$ does not  uniquely define  
the matrices $Q$, $Q_{\perp}$, and $U$
under Assumption \ref{assump:mat_rand_space}  
and hence does not  uniquely define the matrix $M$, but this is immaterial
%% does not affect
for our analysis. 
\end{remark}

\begin{theorem}\label{eqtherr} {\rm [Errors of Alg. \ref{alg1} for an input with a random singular space.]}
%QL0221 
Suppose that $G$ is an $r\times n$ Gaussian  
%%V random 
matrix,
 $H\in\mathbb R^{n\times l}$ is a constant matrix, $n > 36r$, $l>22(r-1)$,   $r \le l < \min (m, n)$,
%%V Assumption \ref{assump:mat_rand_space}. 
and Alg. \ref{alg1} applied  to the matrix $M$ of (\ref{eqassmp}) outputs  two  matrices $X$ and $Y$.

(i) If   $H$ has orthonormal columns,
%%V rows. 
then 
\begin{eqnarray*}
|||M - XY |||/\tilde \sigma_{r+1}(M) \le \sqrt{1 + 16n/l}
\end{eqnarray*}
with a probability no less than
$1 - \exp(-\frac{n}{72}) - \exp(-\frac{l-r}{20})$.

(ii) If   $H$ has full rank $l$, then
\begin{eqnarray*}
|||M - XY |||/\tilde \sigma_{r+1}(M) \le \sqrt{1 + 16\kappa^2(H)n/l}
\end{eqnarray*}
 with a probability no less than
$1 - \exp(-\frac{n}{72}) - \exp(-\frac{l-r}{20})$.
%  Let the matrix $V_1$ in Theorem \ref{thpert1} be the 
%    $n\times r$  Q factor in  a   
% QR factorization  of a normalized $n\times r$ Gaussian matrix 
%$G_{n,r}$ and let
%  
%(i) Then for random variables $\nu = |G_{n, r}|$ and $\mu = |G_{n, r}U_H|$, it holds that
%%\begin{equation}\label{eqerr1} 
%$$ |M-XY|/\bar\sigma_{r+1}(M)\le  \phi_{r,l,n}:=(1+(\nu\mu |H^+|)^2)^{1/2}.$$
% %\end{equation}
% 
%(ii) For $n\ge l\ge r+4\ge 6$, with
%probability no less than $1-2\sqrt{\xi}$ we have
%\begin{equation}
%\phi_{sp,r,l,n}^2 \le 1 + \xi^{-2}~e^2~|H^+|^2\big(\frac{\sqrt{l}(\sqrt{n}+\sqrt{r})}{l-r} \big)^2
%\end{equation}
%and with probability no less than $1-2\sqrt{\xi}$,
%\begin{equation}
%    \phi_{F,r,l,n}^2 \le 1+
%    \xi^{-2}~r^2~\frac{n}{l-r-1}.
%\end{equation}

\end{theorem}

\begin{proof} We can assume that  the matrices $G$ and $GH$ have full rank $r$ (see Thm. \ref{thrnd} and Remark \ref{refllrnk}).
Consider SVD 
\begin{eqnarray*}
M = U\cdot\begin{bmatrix}\Sigma_r & 0\\ 0 & \Sigma_\perp \end{bmatrix} \cdot\begin{bmatrix}Q^T\\Q_\perp^T\end{bmatrix},
\end{eqnarray*}
  write $C_1 := Q^TH$, apply  Cor. \ref{copert0}, and deduce that
$$
|||M - XY |||/\tilde\sigma_{r+1}(M) \le \sqrt{1 + (||C_1^+||_2)^2}.
$$
(i) Recall from Cor. \ref{coro:least_sing_bound} that
\begin{eqnarray}
\prob \big\{  \sigma_r(C_1) = \sigma_r(Q^TH) \ge \frac{1}{4} \sqrt{l/n}  \big\} > 1 - \exp\Big(-\frac{n}{72}\Big) - \exp\Big(-\frac{l-r}{20}\Big).\label{ineq:rand_sub}
\end{eqnarray} 

\noindent (ii) Let $H = U_H\Sigma_HV_H^T$ be a compact SVD. Then
$\sigma_r(C_1)\ge \sigma_r(Q^TU_H)\sigma_r(H)$. 

Similarly to (\ref{ineq:rand_sub}) obtain that
\begin{align*}
\prob \big\{  \sigma_r(C_1) \ge \frac{1}{4} \sqrt{l/n} \cdot\sigma_r(H)  \big\} &\ge \prob \big\{  \sigma_r(Q^TU_H) \ge \frac{1}{4} \sqrt{l/n}\big\}\\
&> 1 - \exp\Big(-\frac{n}{72}\Big) - \exp\Big(-\frac{l-r}{20}\Big).
\end{align*}
Combine Thm.  \ref{thpert1}, equation $||Q_\perp||_2 = 1$, and the bound $\sigma_r(C_1) \ge \frac{1}{4} \sqrt{l/n} \cdot\sigma_r(H)$ and obtain
\begin{eqnarray*}
|||M - XY |||/\tilde \sigma_{r+1}(M) \le \sqrt{1 + ||Q_\perp^THC_1^+ ||_2^2} \le \sqrt{1 + 16\kappa^2(H)n/l}.
\end{eqnarray*}
\end{proof}

 \subsection{Output errors of Alg. \ref{alg1} for a perturbed factor-Gaussian input}\label{serrrang}
 
%------------------------------------------------------------------------------
%------------------------------------------------------------------------------

%QL0127 
\begin{assumption}\label{assmp2} 
For  an $r\times n$ Gaussian  
%%V random 
matrix $G$ and a constant matrix
$A\in \mathbb R^{m\times r}$   of full rank $r< \min (m, n)$, define the matrices
 $B: = \frac{1}{\sqrt{n}}\cdot G$ 
 and  $\Tilde{M}: = AB$ and 
  call $M: = \Tilde{M} + E$ a {\em perturbed  
   right factor-Gaussian 
%%V   random 
   matrix} if the
Frobenius norm of a perturbation matrix $E$ is   sufficiently small in comparison to $\sigma_r(A)$.   
\end{assumption}

%\begin{assumption}
%Let $\Tilde{M} = AB$ be a right $m\times n$ factor Gaussian matrix of rank $r$, let
%$H = U_H\Sigma_HV_H^{*}$ be SVD of an $n\times l$ test matrix,
%and let $\theta = \frac{e\sqrt{l}(\sqrt{n}+\sqrt{r})}{l-r}$ be a constant.
%Define random variables $\nu = ||B||_2$ and 
%$\mu = ||(BU_H)^+||_2$ and recall that 
%$\nu \eqid \nu_{sp,r,n}$ and 
%$\mu \eqid \nu_{sp,r,l}^+$.
%\end{assumption}

\begin{theorem}\label{therrfctr} %{\rm [Errors of Range Finder for a  perturbed right factor-Gaussian matrix.]} 
%QL0221 
Given an $r\times n$  Gaussian  random matrix $G$ and constant matrices
$H\in\mathbb R^{n\times l}$, $A\in\mathbb R^{m\times r}$, and $E \in\mathbb R^{m\times n}$ for $r\le l < \min (m ,n)$,
let $n > 36r$ and $l>22(r-1)$, let $\tilde M$ be a right factor-Gaussian  matrix
%Fix a constant matrix 
%$H\in\mathbb R^{n\times l}$,
%draw an $r\times n$  Gaussian  random matrix $G$ for $r\le l < \min (m, n)$ and let two constant matrices $A$ and $E$  and two random matrices $B$ and 
 %V and $N$ 
 of Assumption \ref{assmp2},
 %V0221 
let $M = \Tilde{M} + E$  for a perturbation matrix $E$,
and let Alg. \ref{alg1}  applied to the  matrix $M$ output two matrices
 $X$ and $Y$. 

(i) If the matrix $H$ has orthonormal columns 
and  if $||E||_F \le \frac{\sigma_r(A)}{48\sqrt{n/l} + 6}$, then
%QL0221
%$$||M-XY||_2\le 
% \sqrt{1+ 100n/l} \cdot\sigma_{r+1}(M) 
%$$
$$
|||M-XY|||/\tilde \sigma_{r+1}(M) \le \sqrt{1+ 100~n/l}
$$
with a probability no less than $1 - \exp(-\frac{l-r}{20}) - \exp(-\frac{n-r}{20}) - \exp(-\frac{n}{72})$.
  
(ii) Let $\kappa(H) = ||H||_2||H^+||_2$ denote the spectral condition number of $H$. If $H$ has full rank and if  $||E||_F \le \frac{\sigma_r(A)}{12}\min(1,  \frac{1}{4\sqrt{n/l}\cdot\sigma_l(H) + 0.5})$,
 then
%$$
%||M - XY||_2\le \sqrt{1 + 100\kappa^2(H)n/l }\cdot \sigma_{r+1}(M),
%$$
$$
|||M-XY|||/\tilde \sigma_{r+1}(M)\le \sqrt{1 + 100~\kappa^2(H)n/l } 
$$ 
with a probability no less than $1 - \exp(-\frac{l-r}{20}) - \exp(-\frac{n-r}{20}) - \exp(-\frac{n}{72}).$

\end{theorem} 

%\begin{proof}
%We prove claim (i) in the next subsection.
%Let us deduce claim (ii).
%Lemma \ref{appendix:lemma:MuNuBound} ensures that $ \nu\mu \le \theta/\xi
%     ~\textrm{ and thus }~
%     \phi \ge \xi/2||H^+||\theta$
% with a probability no less than $1-2\sqrt{\xi}$.
%Substitute this inequality into the bound of claim (i) and deduce claim (ii) of the theorem.
% \end{proof} 

%----------------------------------------------------
%
% Proof of theorem 4.2
%\input{proof.tex}

%\subsection{Proof of claim (i) of Theorem \ref{therrfctr}}\label{sprf}  

%------------------------------------------------------------------------------

\begin{proof}

Let  the matrices $B$, $AB$, and $BH$  have full rank 
(see Thm. \ref{thrnd}  and Remark \ref{refllrnk}) and let
\begin{eqnarray*}
M = \begin{bmatrix}U_r & U_{\perp} \end{bmatrix} \begin{bmatrix}\Sigma_r & 0 \\ 0 & \Sigma_\perp \end{bmatrix} \begin{bmatrix}V_r^T \\ V_\perp^T\end{bmatrix} 
\textrm{ and }
\tilde M = \begin{bmatrix}\tilde U_r & \tilde U_{\perp} \end{bmatrix} \begin{bmatrix}\tilde \Sigma_r & 0 \\ 0 & 0 \end{bmatrix} \begin{bmatrix}\tilde V_r^T \\ \tilde V_\perp^T\end{bmatrix} 
\end{eqnarray*}
be SVDs, where
%V and thus 
$V_r$ and $\tilde V_r$ are the matrices of the $r$ top right singular vectors of $M$ and $\tilde M$, respectively.
Define $C_1 = V_r^TH$ and $\tilde C_1 := \tilde V_r^TH$ as in (\ref{eqc12}). 
Now Thm. \ref{thpert1} 
implies that
%\begin{equation}
%|| M - XY ||_2 / \sigma_{r+1}(M) \le \sqrt{1 + ||V_\perp^THC_1^+ ||_2^2}\label{ineq:fact_gauss_pert}
%\end{equation}
\begin{eqnarray*}
||| M - XY ||| /\tilde \sigma_{r+1}(M) \le \sqrt{1 + ||V_\perp^THC_1^+ ||_2^2}\label{ineq:fact_gauss_pert}.
\end{eqnarray*}

 Next we prove {\bf claim (i)}. Since $||V_\perp||_2 = ||H||_2 = 1$,  we only need to  deduce that $\sigma_r(C_1) > \frac{1}{10}\sqrt{l/n}$ {\em whp}.
  
Recall that the columns of $V_r$ span the row space of an $r\times n$ Gaussian  
%%V random 
matrix $G$ and deduce from 
 Cor.  \ref{coro:least_sing_bound}  that $\sigma_r(\tilde C_1) > \frac{1}{4}\sqrt{l/n}$ with a probability no less than $1 - \exp(-\frac{l-r}{20}) - \exp(-n/72)$.
It remains to verify that {\em whp} the perturbation $E$ only slightly alters the leading right  singular space of $\tilde M$  and that $\sigma_r(C_1)$ is close to $\sigma_r(\tilde C_1)$.
 
  If the norm of the perturbation matrix $||E||_F$ is sufficiently small, then
  by virtue of Lemma \ref{lemma:pert_sing_space} there exists a matrix $P$ such that the matrices
   $\tilde V_r + \tilde V_\perp P$ and $V_r$ have the same column space and that furthermore $||P||_F \le \frac{2||E||_F}{\sigma_r(\tilde M) - 2||E||_F }$. This implies a desired bound on the differences of the smallest positive singular values of $C_1$ and $\tilde C_1$; next we  supply the details.
 
%QL0221 rewrote this sentence
Claim (ii) of Thm. \ref{thsiguna} 
%Vclaim (ii) 
%V we 
 %V0221 suggests
implies that  {\em whp} the $r$-th top singular value of $\tilde M = AB$ is not much less than the $r$-th top singular value of $A$.
Readily  deduce
%V0221, consider simple probabilistic bound 
from Cor. \ref{coro:least_sing_bound} that
%QL0221
\begin{eqnarray*} 
\prob \{ \sigma_r(\tilde M) < \sigma_r(A)/3 \} \le \prob \{\sigma_r(B) = \frac{1}{\sqrt{n}} \sigma_r(G) < 1/3\} \le e^{-(n-r)/20}.
\end{eqnarray*}
%V0221and t
Hence $||E||_F \le \frac{\sigma_r(A)}{48\sqrt{n/l} + 6}\le \frac{\sigma_r (\tilde M) }{16\sqrt{n/l} + 2}$ with a probability no less than $1 - e^{-(n-r)/20}$, and so  $||P||_2 \le \frac{1}{8}\sqrt{l/n}$ for some  matrix $P$ of Lemma \ref{lemma:pert_sing_space} such that  $||P||_2 \le \frac{1}{8}\sqrt{l/n}$.  

Now let this holds, let  $\sigma_r(\tilde C_1) > \frac{1}{4}\sqrt{l/n}$, and deduce that
\begin{align}
\sigma_r(V_r^TH) &= \sigma_r\big((I_r + P^TP)^{-1/2}(\tilde V_r^T + P^T\tilde V_\perp^T)H\big) \label{eq:pf_of_thm} \\ 
&\ge \sigma_r\big((I_r + P^TP)^{-1/2}\big) \sigma_r(\tilde V_r^TH + P^T\tilde V_\perp^TH)\\
&\ge \frac{1}{\sqrt{1 + (\sigma_1(P))^2}} \big(  \sigma_r(\tilde C_1) - \sigma_1(P)\big) > \frac{1}{10} \sqrt{l/n}. \label{ineq:pf_of_thm2}
\end{align}
\noindent Equality (\ref{eq:pf_of_thm}) holds because the matrix $\tilde V_r + \tilde V_\perp P$ is normalized by $(I_r + P^TP)^{-1/2}$ (see Remark \ref{remark:pert_sing_space}) and has the same column span as $V_r$.
By applying the union bound deduce that inequality (\ref{ineq:pf_of_thm2}) holds with a probability no less than $1 - \exp(-\frac{l-r}{20}) - \exp(-\frac{n-r}{20}) - \exp(-\frac{n}{72})$.

To prove {\bf claim (ii)}, we essentially need to show that $\sigma_r(C_1) = \sigma_r(V_r^TH) \ge \frac{1}{10}\sqrt{l/n}\cdot \sigma_l(H)$, and then the claim will follow readily from inequality (\ref{ineq:fact_gauss_pert}).  
Let $H = U_H\Sigma_HV_H^T$ be compact SVD
such that $U_H\in \mathbb R^{n\times l}$ , $\Sigma_H\in \mathbb R^{l\times l}$, and $V_H\in \mathbb R^{l\times l}$, and obtain that $\sigma_r(\tilde C_1) \ge \sigma_r(\tilde V_r^TU_H)\sigma_l(\Sigma_H)$ and 
\begin{eqnarray*}
\prob \{ \sigma_r(\tilde C_1) < \frac{1}{4}\sqrt{l/n}\cdot\sigma_l(H)\} < \exp(-\frac{l-r}{20}) + \exp(-\frac{n}{72}).
\end{eqnarray*}
Next  bound $\sigma_r(C_1)$ by showing that  the column spaces of $V_r$ and $\tilde V_r$ are sufficiently close 
to one another
%%V
if the perturbation $V_r-\tilde V_r$ is sufficiently small. 
Assume that $||E||_F \le \frac{\sigma_r(A)}{12}$, and  then the assumptions of Lemma \ref{lemma:pert_sing_space} hold whp.
By applying the same argument as in the proof of claim (i), deduce that 
\begin{eqnarray*}
||E||_F \le \min\Big( \frac{\sigma_r (\tilde M)}{4}, \frac{\sigma_r (\tilde M) }{16\sqrt{n/l}\cdot\sigma_l(H) + 2}\Big) 
\end{eqnarray*}
with a probability no less than $1 - e^{-(n-r)/20}$.
It follows that  $||P||_2 \le \frac{1}{8}\sqrt{l/n}\cdot\sigma_l(H)$ for some matrix $P$ of Lemma \ref{lemma:pert_sing_space}. Hence $\sigma_r(C_1)\ge \frac{1}{10}\sqrt{l/n}\cdot \sigma_l(H)$ whp.
\end{proof}

%----------------------------------------------------------- 
%----------------------------------------------------------------------------
 
\section{Numerical tests}\label{srndsmpl}

%----------------------------------------------------------------------------
 
  In this section we cover our tests  
 of dual sublinear cost 
  variants of Alg. \ref{alg1}.
  The  standard normal distribution function randn of MATLAB
has been applied  to generate Gaussian matrices. 
 The MATLAB function "svd()"  has been
 applied  to calculate
 the $\epsilon$-rank  
for  $\epsilon=10^{-6}$. 
%edits by John begin NEW TEXT
%John Svadlenka 
The tests for Tables \ref{SuperfastTable}--\ref{tab8.10}  have been performed
on a 64-bit Windows machine with an Intel i5 dual-core 1.70 GHz processor by using custom programmed software in $C^{++}$ and compiled with LAPACK version 3.6.0 libraries.
%edits by John end NEW TEXT

%----------------------------------------------------------------------------
  
\subsection{Input matrices for {\em LRA} tests}\label{ststmtrcs}

% - - - - - - - - - - - - - - - - - - - - - - - - - - - - - - - - - - - - -
   
 We generated the following classes of 
 input matrices $M$ for testing {\em LRA} algorithms.
\medskip

{\bf Class I:} $M=U_M\Sigma_M V_M^*$,
where $U_M$ and $V_M$ are the Q factors
of the thin QR orthogonalization of  
$n\times n$ Gaussian matrices, 
$\Sigma_M=\diag(\sigma_j)_{j=1}^n$; 
 $\sigma_j=1/j,~j=1,\dots,r$,
$\sigma_j=10^{-10},~j=r+1,\dots,n$
(cf. [H02, Sec. 28.3]), 
 and  $n=256,
512,
1024$.
(Hence $||M||_2=1$ and 
$||M^+||_2=10^{10}$.)    

\medskip

{\bf Class II:}   
(i) The matrices $M$ of   the discretized single-layer Laplacian operator of  \cite[Sec. 7.1]{HMT11}:
%\begin{equation}
$[S\sigma](x) = c\int_{\Gamma_1}\log{|x-y|}\sigma(y)dy,x\in\Gamma_2$,
%\end{equation}
for two circles $\Gamma_1 = C(0,1)$ and $\Gamma_2 = C(0,2)$  on the complex plane.
We arrived at the matrices      $M=(m_{ij})_{i,j=1}^n$,  
%\begin{equation}
$m_{i,j} = c\int_{\Gamma_{1,j}}\log|2\omega^i-y|dy$ 
for a constant $c$,  $||M||=1$ and
%\end{equation} 
 the arc $\Gamma_{1,j}$  of  $\Gamma_1$ defined by
the angles in the range $[\frac{2j\pi}{n},\frac{2(j+1)\pi}{n}]$.

(ii) The matrices  that approximate the inverse of a large sparse
matrix obtained from a finite-difference operator
of  \cite[Sec. 7.2]{HMT11}.

\medskip

{\bf Class III:} The dense  matrices of five classes with smaller ratios of ``numerical rank/$n$"  from the built-in test problems in Regularization 
Tools, which came from discretization (based on Galerkin or quadrature methods) of the Fredholm  Integral Equations of the first kind:\footnote{See 
%database at 
 http://www.math.sjsu.edu/singular/matrices and 
  http://www2.imm.dtu.dk/$\sim$pch/Regutools 
  
For more details see Chapter 4 of the Regularization Tools Manual at \\
  http://www.imm.dtu.dk/$\sim$pcha/Regutools/RTv4manual.pdf }

\medskip

{\em baart:}       Fredholm Integral Equation of the first kind,

{\em shaw:}        one-dimensional image restoration model,
 
{\em gravity:}     1-D gravity surveying model problem,
 
%heat:        inverse heat equation.

%parallax:    Stellar parallax problem with 28 fixed, real observations.

%tomo:        a 2D tomography test problem. 

%ursell:      integral equation with no square integrable solution.

wing:        problem with a discontinuous
 solution,

{\em foxgood:}     severely ill-posed problem.

We used $1024\times 1024$ SVD-generated input matrices of class I having numerical rank $r = 32$,  $400 \times 400$ Laplacian input matrices of class II(i)
%John's edits START
%having numerical rank $r = 3$, 
having numerical rank $r = 36$,
%John's edits END
 $408 \times 800$ matrices having numerical rank $r = 145$  and
representing finite-difference inputs  of class II(ii),
and $1000 \times 1000$ matrices of class III (from the San Jose University database), having numerical rank 4, 6, 10, 12, and 25
for the matrices of the classes {\em wing, baart, foxgood, shaw}, and {\em gravity}, respectively.

%John's edits START
%Then again we repeated the tests 1000 times for each class of input matrices and

%----------------------------------------------------------------------------
%------------------------------------------------------------------------------

%\clearpage 

\subsection{Five families of Ultrasparse sketch matrices $H$}\label{s17m} 

%John's edits START 
We generated our  $n\times (r+p)$  sketch matrices $H$ for random $p=1,2, \dots, 21$ by using
%John's edits END
%  \item%1
{\em  3-ASPH,
%\item%2
  3-APH (see Appendix \ref{spreprmlt}), and
%\item%3 
 Random permutation matrices.}
  When  the overestimation parameter $p$ was considerable,
we actually computed {\em LRA} of numerical  rank larger than $r$, and so {\em LRA} was frequently  
closer to an input matrix than 
the optimal 
rank-$r$ approximation. Accordingly,
 the output error norms in our tests ranged from about $10^{-4}$ to $10^{4}$ {\em  relative to the optimal errors}.
%\end{itemize}

We  obtained every 3-APH and every 3-ASPH
matrix by applying three Hadamard's recursive steps
(\ref{eqrfd}) followed by random column permutation defined by random permutation of the integers from 1 to $n$  inclusive. While generating a 3-ASPH matrix we also applied random scaling  with a diagonal matrix $D=\diag(d_i)_{i=1}^n$ where we have 
chosen the values of
independent identically distributed  {\em (iid)}  random variables $d_i$ sampled under the uniform
probability distribution  from the set
%John's edits START 
%$\{4, -3, -2, -1, 0, 1 ,2, 3, ,4\}$.
$\{-4, -3, -2, -1, 0, 
1 ,2, 3, 4\}$.
%John's edits END  
 
%We permuted all inverses of bidiagonal matrices except for Class 5 of multipliers.
We used the following families of sketch matrices $H$:
%\begin{itemize}
%\item Family 
(0)	Gaussian (for control),
%\item Family ]
(1)	sum of a 3-ASPH  and a permutation matrix,
%\item Family 
(2) sum of a 3-ASPH  and two permutation matrices,
%\item Family 
(3)	sum of a 3-ASPH  and three permutation matrices,
%\item Family 
(4)	sum of a 3-APH  and three  permutation matrices, and
%\item Family 
(5)	sum of a 3-APH  and two permutation matrices.
%\end{itemize}

%----------------------------------------------------------------------------
  
\subsection{Test results}\label{ststrslts}

 Tables
 \ref{SuperfastTable}--\ref{tab8.13} 
 display the average relative error norm  $ \frac{\|M - \tilde M \|_2}{\|M - M_{nrank}\|_2}$ in  our tests repeated 100 times for each class of input matrices and
%John's edits END
 each 
size of an input matrix and sketch matrix $H$ for Alg.  \ref{alg1} or    
 for each size of an input matrix and a pair of left-hand and right-hand sketch matrices $F$ and $H$  for Alg.  \ref{alg01}.
 
% - - - - - - - - - - - - - - - - - - - - - -

In all our tests we applied the sketch matrices of the six  families of the previous subsection. 

 %{tab8.9}  
   
%John's edits START June 7   
 Tables  \ref{SuperfastTable}--\ref{tab8.10} display the average relative error norm for the output of
 Alg. \ref{alg1}; in  our tests 
 it ranged from about $10^{-3}$ to $10^{1}$.
% with the exception of multiplier families 13--17 for the inverse Laplace input
% matrix, in which case the range was from about $10^{-3$to $10^{-5}$.
% \\
%John's edits END  June 7
 The numbers in parentheses in the first line of Tables \ref{tab8.9}
 and  \ref{tab8.10} show the numerical rank of input matrices.

%John's edits START  June 7
 Tables  \ref{tab8.11}--\ref{tab8.13} display the average relative error norm for the output of
 Alg. \ref{alg01} 
applied  to the same input matrices from classes I--III as in our experiments for Alg. \ref{alg1}.

In these tests we used 
   $n\times \ell$ and $\ell\times m$ sketch matrices for $\ell = r+p$ and $k=c\ell$ for $c=1,2,3$ and random $p=1,2, \dots, 21$.
%John's edits END  June 7

\begin{table}[htb]
\begin{center} 
\begin{tabular}{|c|c|c|c|c|c|c|}
\hline
			& \multicolumn{2}{|c|}{SVD-generated Matrices} & \multicolumn{2}{|c|}{Laplacian Matrices} & \multicolumn{2}{|c|}{Finite Difference Matrices}\\\hline
 \text{Family No.} & \text{Mean} & \text{Std} & \text{Mean} & \text{Std} & \text{Mean} & \text{Std} \\\hline
%John's edits START  June 7
% SVD-generated Matrices updated results
Family 0 & 4.52e+01 & 5.94e+01 & 6.81e-01 & 1.23e+00 & 2.23e+00 & 2.87e+00\\\hline
Family 1 & 3.72e+01 & 4.59e+01 & 1.33e+00 & 2.04e+00 & 8.22e+00 & 1.10e+01\\\hline
Family 2 & 5.33e+01 & 6.83e+01 & 1.02e+00 & 2.02e+00 & 4.92e+00 & 4.76e+00\\\hline
Family 3 & 4.82e+01 & 4.36e+01 & 7.56e-01 & 1.47e+00 & 4.82e+00 & 5.73e+00\\\hline
Family 4 & 4.68e+01 & 6.65e+01 & 7.85e-01 & 1.31e+00 & 3.53e+00 & 3.68e+00\\\hline
Family 5 & 5.45e+01 & 6.23e+01 & 1.03e+00 & 1.78e+00 & 2.58e+00 & 3.73e+00\\\hline
%John's edits END  June 7
\end{tabular}
\caption{Relative error norms in tests for matrices of classes I and II}
 \label{SuperfastTable} 
\end{center}
\end{table}

%------------------------------------------------------------------------------

%We executed our  experiments
%for these classes of matrices
%on a 64-bit Windows machine with an Intel i5 dual-core 1.70 GHz processor using custom programmed software in $C^{++}$ and %compiled with LAPACK version 3.6.0 libraries.

\begin{table}
\begin{center}
\begin{tabular}{|c|c|c|c|c|c|c|}
\hline
			& \multicolumn{2}{|c|}{wing (4)} & \multicolumn{2}{|c|}{baart (6)} 
\\\hline
  
 \text{Family No.} & \text{Mean} & \text{Std} & \text{Mean}  & \text{Std}   \\\hline
%John's edits START  June 7 
 Family 0 	& 1.07e-03 & 6.58e-03 & 2.17e-02 & 1.61e-01 
\\\hline
 Family 1 	& 3.54e-03 & 1.39e-02  & 1.37e-02 & 6.97e-02 
\\\hline
 Family 2 	& 4.74e-03 & 2.66e-02  & 1.99e-02 & 8.47e-02  
\\\hline
 Family 3 	& 1.07e-03 & 5.69e-03  & 1.85e-02 & 8.74e-02  
\\\hline
 Family 4 	& 4.29e-03 & 1.78e-02  & 8.58e-03 & 5.61e-02  
\\\hline
 Family 5 	& 1.71e-03 & 1.23e-02  & 3.66e-03 & 2.38e-02 
\\\hline
%John's edits END  June 7
\end{tabular}
\caption{Relative error norms for  input matrices of class III (of San Jose University database)}\label{tab8.9}
\end{center}
\end{table}

\begin{table}[htb] 
\begin{center}
\begin{tabular}{|c|c|c|c|c|c|c|}
\hline
			& \multicolumn{2}{|c|}{foxgood (10)} & \multicolumn{2}{|c|}{shaw (12)} & \multicolumn{2}{|c|}{gravity (25)}\\\hline
 
 \text{Family No.} & \text{Mean} & \text{Std} & \text{Mean}  & \text{Std} & \text{Mean} 	& \text{Std} \\\hline
 %John's edits START  June 7 
 Family 0 	& 1.78e-01 & 4.43e-01  & 4.07e-02 & 1.84e-01  & 5.26e-01 & 1.24e+00\\\hline
 Family 1 	& 1.63e+00 & 3.43e+00  & 8.68e-02 & 3.95e-01  & 3.00e-01 & 7.64e-01\\\hline
 Family 2 	& 1.97e+00 & 4.15e+00  & 7.91e-02 & 4.24e-01  & 1.90e-01 & 5.25e-01\\\hline
 Family 3 	& 1.10e+00 & 2.25e+00  & 4.50e-02 & 2.21e-01  & 3.63e-01 & 1.15e+00\\\hline
 Family 4 	& 1.23e+00 & 2.11e+00  & 1.21e-01 & 5.44e-01  & 2.36e-01 & 5.65e-01\\\hline
 Family 5 	& 1.08e+00 & 2.32e+00  & 1.31e-01 & 5.42e-01  & 2.66e-01 & 8.22e-01\\\hline
 %John's edits END  June 7 
\end{tabular}
\caption{Relative error norms for input matrices of class III (of San Jose University database)} \label{tab8.10}
\end{center}
\end{table}

%John's edits START  June 7 

\begin{table}[htb]
\begin{center}
\begin{tabular}{|c|c|c|c|c|c|c|c|}
\hline
		&	& \multicolumn{2}{|c|}{SVD-generated Matrices} & \multicolumn{2}{|c|}{Laplacian Matrices} & \multicolumn{2}{|c|}{Finite Difference Matrices}\\\hline	
 
\text{$k$} & \text{Class No.} & \text{Mean} & \text{Std} & \text{Mean}  & \text{Std} & \text{Mean} 	& \text{Std} \\\hline

\multirow{6}*{$\ell$} 
& Family 0 & 2.43e+03 & 1.19e+04 & 1.28e+01 & 2.75e+01 & 9.67e+01 & 1.48e+02 \\ \cline{2-8}
& Family 1 & 1.45e+04 & 9.00e+04 & 8.52e+03 & 8.48e+04 & 7.26e+03 & 2.47e+04 \\ \cline{2-8}
& Family 2 & 4.66e+03 & 2.33e+04 & 3.08e+01 & 4.07e+01 & 3.80e+02 & 1.16e+03 \\ \cline{2-8}
& Family 3 & 2.82e+03 & 9.47e+03 & 2.42e+01 & 3.21e+01 & 1.90e+02 & 3.90e+02 \\ \cline{2-8}
& Family 4 & 3.15e+03 & 7.34e+03 & 2.71e+01 & 4.69e+01 & 1.83e+02 & 2.92e+02 \\ \cline{2-8}
& Family 5 & 2.40e+03 & 6.76e+03 & 2.01e+01 & 3.56e+01 & 2.31e+02 & 5.33e+02  \\\hline

\multirow{6}*{$2\ell$}
& Family 0 & 5.87e+01 & 5.59e+01 & 7.51e-01 & 1.33e+00 & 3.17e+00 & 3.89e+00 \\ \cline{2-8}
& Family 1 & 7.91e+01 & 9.86e+01 & 3.57e+00 & 7.07e+00 & 1.55e+01 & 2.39e+01 \\ \cline{2-8}
& Family 2 & 5.63e+01 & 3.93e+01 & 3.14e+00 & 4.50e+00 & 5.25e+00 & 5.93e+00 \\ \cline{2-8}
& Family 3 & 7.58e+01 & 8.58e+01 & 2.84e+00 & 3.95e+00 & 4.91e+00 & 6.03e+00 \\ \cline{2-8}
& Family 4 & 6.24e+01 & 4.54e+01 & 1.99e+00 & 2.93e+00 & 3.64e+00 & 4.49e+00 \\ \cline{2-8}
& Family 5 & 6.41e+01 & 6.12e+01 & 2.65e+00 & 3.13e+00 & 3.72e+00 & 4.54e+00  \\\hline

\multirow{6}*{$3\ell$}
& Family 0 & 9.29e+01 & 3.29e+02 & 8.33e-01 & 1.54e+00 & --- & --- \\ \cline{2-8}
& Family 1 & 5.58e+01 & 4.20e+01 & 3.09e+00 & 4.08e+00 & --- & --- \\ \cline{2-8}
& Family 2 & 5.11e+01 & 4.94e+01 & 1.70e+00 & 2.08e+00 & --- & --- \\ \cline{2-8}
& Family 3 & 6.70e+01 & 8.27e+01 & 2.35e+00 & 2.96e+00 & --- & --- \\ \cline{2-8}
& Family 4 & 5.36e+01 & 5.74e+01 & 2.14e+00 & 3.76e+00 & --- & --- \\ \cline{2-8}
& Family 5 & 4.79e+01 & 4.58e+01 & 1.81e+00 & 2.94e+00 & --- & --- \\\hline

\end{tabular}
\caption{Relative error norms in tests for matrices of classes I and II }
\label{tab8.11}
\end{center}
\end{table}

\begin{table}[htb] 
\begin{center}

%John's edits START  June 22 
\begin{tabular}{|c|c|c|c|c|c|}
\hline
		&	& \multicolumn{2}{|c|}{wing (4)} & \multicolumn{2}{|c|}{baart (6)} \\\hline
 
\text{$k$} & \text{Class No.} & \text{Mean} & \text{Std} & \text{Mean}  & \text{Std}  \\\hline
\multirow{6}*{$\ell$}
& Family 0 & 1.70e-03 & 9.77e-03 & 4.55e+00 & 4.47e+01  \\ \cline{2-6}
& Family 1 & 3.58e+02 & 3.58e+03 & 1.42e-01 & 9.20e-01  \\ \cline{2-6}
& Family 2 & 2.16e-01 & 2.10e+00 & 1.10e-02 & 6.03e-02  \\ \cline{2-6}
& Family 3 & 7.98e-04 & 7.22e-03 & 4.14e-03 & 3.41e-02  \\ \cline{2-6}
& Family 4 & 5.29e-03 & 3.57e-02 & 2.22e+01 & 2.21e+02  \\ \cline{2-6}
& Family 5 & 6.11e-02 & 5.65e-01 & 3.33e-02 & 1.30e-01  \\\hline

\multirow{6}*{$2\ell$}
& Family 0 & 7.49e-04 & 5.09e-03 & 5.34e-02 & 2.19e-01  \\ \cline{2-6}
& Family 1 & 4.74e-03 & 2.32e-02 & 2.14e-02 & 1.33e-01  \\ \cline{2-6}
& Family 2 & 3.01e-02 & 2.34e-01 & 1.26e-01 & 7.86e-01  \\ \cline{2-6}
& Family 3 & 2.25e-03 & 1.38e-02 & 5.91e-03 & 2.63e-02  \\ \cline{2-6}
& Family 4 & 3.94e-03 & 2.54e-02 & 1.57e-02 & 6.71e-02  \\ \cline{2-6}
& Family 5 & 2.95e-03 & 1.47e-02 & 1.58e-02 & 1.20e-01  \\\hline

\multirow{6}*{$3\ell$}
& Family 0 & 4.59e-03 & 2.35e-02 & 1.50e-02 & 7.09e-02  \\ \cline{2-6}
& Family 1 & 5.96e-03 & 2.82e-02 & 7.57e-03 & 4.84e-02  \\ \cline{2-6}
& Family 2 & 1.74e-02 & 1.06e-01 & 6.69e-03 & 2.97e-02  \\ \cline{2-6}
& Family 3 & 3.07e-03 & 3.07e-02 & 1.16e-02 & 5.16e-02  \\ \cline{2-6}
& Family 4 & 2.57e-03 & 1.47e-02 & 2.35e-02 & 9.70e-02  \\ \cline{2-6}
& Family 5 & 4.32e-03 & 2.70e-02 & 1.36e-02 & 5.73e-02  \\\hline

%John's edits END  June 22 

\end{tabular}
\caption{Relative error norms for  input matrices of class III (of San Jose University database) }
\label{tab8.12}
\end{center}

\end{table}

\begin{table}[htb]
\begin{center}
\begin{tabular}{|c|c|c|c|c|c|c|c|}
\hline
& & \multicolumn{2}{|c|}{foxgood (10)} & \multicolumn{2}{|c|}{shaw (12)} & \multicolumn{2}{|c|}{gravity (25)}\\\hline
 
\text{$k$} & \text{Class No.} & \text{Mean} & \text{Std} & \text{Mean}  & \text{Std} & \text{Mean} 	& \text{Std} \\\hline
\multirow{6}*{$\ell$}
& Family 0 & 5.46e+00 & 1.95e+01 & 8.20e-01 & 4.83e+00 & 8.56e+00 & 3.33e+01 \\ \cline{2-8}
& Family 1 & 8.51e+03 & 1.88e+04 & 1.12e+00 & 5.75e+00 & 1.97e+01 & 1.00e+02 \\ \cline{2-8}
& Family 2 & 5.35e+03 & 1.58e+04 & 1.93e-01 & 1.51e+00 & 8.79e+00 & 4.96e+01 \\ \cline{2-8}
& Family 3 & 6.14e+03 & 1.74e+04 & 4.00e-01 & 1.90e+00 & 7.07e+00 & 2.45e+01 \\ \cline{2-8}
& Family 4 & 1.15e+04 & 2.33e+04 & 2.95e-01 & 1.71e+00 & 4.31e+01 & 3.80e+02 \\ \cline{2-8}
& Family 5 & 7.11e+03 & 1.87e+04 & 2.18e-01 & 9.61e-01 & 6.34e+00 & 2.59e+01\\\hline

\multirow{6}*{$2\ell$}
& Family 0 & 2.70e-01 & 7.03e-01 & 5.54e-02 & 2.62e-01 & 5.34e-01 & 1.59e+00 \\ \cline{2-8}
& Family 1 & 5.24e+02 & 5.19e+03 & 4.67e-02 & 2.35e-01 & 1.38e+01 & 1.31e+02 \\ \cline{2-8}
& Family 2 & 2.45e+00 & 3.47e+00  & 8.31e-02 & 6.37e-01 & 5.47e-01 & 1.69e+00 \\ \cline{2-8}
& Family 3 & 2.43e+00 & 3.74e+00  & 1.24e-01 & 8.52e-01 & 5.10e-01 & 1.24e+00 \\ \cline{2-8}
& Family 4 & 2.17e+00 & 2.92e+00  & 1.76e-01 & 8.76e-01 & 2.60e-01 & 7.38e-01 \\ \cline{2-8}
& Family 5 & 2.10e+00 & 3.34e+00 & 1.26e-01 & 5.99e-01 & 5.68e-01 & 1.46e+00 \\\hline

\multirow{6}*{$3\ell$}
& Family 0 & 2.62e-01 & 8.16e-01 & 4.49e-02 & 1.64e-01 & 4.59e-01 & 1.38e+00 \\ \cline{2-8}
& Family 1 & 2.72e+00 & 4.60e+00 & 6.84e-02 & 3.43e-01 & 3.44e-01 & 8.60e-01 \\ \cline{2-8}
& Family 2 & 2.42e+00 & 3.92e+00 & 8.26e-02 & 5.38e-01 & 6.89e-01 & 2.15e+00 \\ \cline{2-8}
& Family 3 & 3.22e+02 & 3.20e+03 & 6.06e-02 & 2.95e-01 & 5.26e-01 & 1.17e+00 \\ \cline{2-8}
& Family 4 & 1.91e+00 & 3.36e+00 & 6.61e-02 & 3.36e-01 & 6.19e-01 & 1.54e+00 \\ \cline{2-8}
& Family 5 & 2.73e+00 & 6.90e+00 & 5.72e-02 & 2.39e-01 & 7.22e-01 & 1.59e+00 \\\hline

\end{tabular}
\caption{Relative error norms for  input matrices of class III (of San Jose University database)}
\label{tab8.13}
\end{center}

\end{table}
\clearpage
%John's edits END  June 7 

%----------------------------------------------------------- 
%------------------------------------------------------------------------------
 
%\section{Conclusions}\label{scncl}

%\clearpage
\bigskip

%\bigskip
%\bigskip

%------------------------------------------------------------------------------
 
{\bf \Large Appendix} 
\appendix

%------------------------------------------------------------------------------

\section{The spectral norms of a Gaussian  matrix and its pseudo inverse}\label{snrmg}

 %------------------------------------------------------------------------------
Hereafter
$\Gamma(x)=
\int_0^{\infty}\exp(-t)t^{x-1}dt$
denotes the Gamma function; 
%QL0127 
% Introducing random variable $\nu$ here.
%V In the following Theorems in this section,  
 $\nu_{p, q}$ and $\nu^+_{p, q}$ denote the random variables 
representing the spectral norms of a $p\times q$ Gaussian random matrix and its Moore-Penrose pseudo inverse, respectively.

%-----------------------------------------------------------------------------

\begin{theorem}\label{thsignorm} {\rm [Spectral norms of a Gaussian matrix. 
%------------------------------------------------------------------------------ 
 See  \cite[Thm. II.7]{DS01}.]}

  {\rm Probability}$\{\nu_{m,n}>t+\sqrt m+\sqrt n\}\le
\exp(-t^2/2)$ for $t\ge 0$, 
 $\mathbb E(\nu_{m,n})\le \sqrt m+\sqrt n$.
\end{theorem}

%------------------------------------------------------------------------------

\begin{theorem}\label{thsiguna} 
 {\rm [Spectral norms of the pseudo inverse of a Gaussian matrix.]} 
%{\rm and}~\zeta(t)= 
%\frac{\sqrt{2m}}{\Gamma(m/2)}(t\sqrt{m/2})^{m-1}\exp(-mt^2/2)=
%2~t^{m-1}(\frac{m}{2})^{m/2}{\rm exp}(-\frac{m}{2}t^2)/\Gamma(\frac{m}{2}).$$ 

(i)  {\rm Probability} $\{\nu_{m,n}^+\ge m/x^2\}<\frac{x^{m-n+1}}{\Gamma(m-n+2)}$
for $m\ge n\ge 2$ and all positive $x$,

(ii)  {\rm Probability} $\{\nu_{m,n}^+\ge  	 
  t\frac{e\sqrt{m}}{m-n+1}\}\le t^{n-m}$
  for all $t\ge 1$ provided that $m\ge 4$, 
  
(iii) $\mathbb E(\nu^+_{m,n})\le \frac{e\sqrt{m}}{m-n}$ 
provided that $m\ge n+2\ge 4$,
 
%QL0127 Since we did not use this result throughout the paper,
% I suggest we leave this result out.
%(iv) {\rm Probability} $\{\nu_{{\rm sp},n,n}^+\ge x\}\le \frac{2.35\sqrt n}{x}$ 
%for $n\ge 2$  and all positive $x$, and furthermore $||M_{n,n}+G_{n,n}||^+\le \nu_{n,n}$
%for any $n\times n$ matrix $M_{n,n}$ and an $n\times n$ Gaussian matrix $G_{n,n}$. 
\end{theorem}

%------------------------------------------------------------------------------

\begin{proof}
 See \cite[Proof of Lemma 4.1]{CD05} for claim (i),   
\cite[Prop.
 10.4 and Eqns. (10.3) and (10.4)]{HMT11} for claims (ii) and (iii), 
and \cite[Thm. 3.3]{SST06} for claim (iv).
\end{proof}
 
%------------------------------------------------------------------------------
%------------------------------------------------------------------------------

Thm. \ref{thsiguna}
implies reasonable probabilistic upper bounds on the norm 
 $\nu_{m,n}^+$,
even where the integer $|m-n|$ is close to 0; {\em whp} the upper bounds of Thm. \ref{thsiguna}
on the norm $\nu^+_{m,n}$ decrease very fast as the difference $|m-n|$ grows from 1.

%------------------------------------------------------------------------------
 
%------------------------------------------------------------------------------

\section{Randomized  pre-processing of lower rank matrices}\label{srndoprpr}

%------------------------------------------------------------------------------

%QL0419 added "greater or equal to" statistical ordering
Hereafter 
$A  \preceq B$ ($A \succeq B$) means that $A$ is  statistically 
less (greater) or equal to $B$.
The following  theorem 
(cf. \cite[Sec. 8.2]{PLSZa})  shows that
  pre-processing with 
Gaussian sketch matrices $X$ and/or $Y$ transforms any matrix that admits $LRA$
into a perturbation of a factor-Gaussian matrix.
%vp

\begin{theorem}\label{thquasi}  
%VP
For five integers $k$, $l$, $m$, $n$, and  $r$
satisfying the bounds
$r\le k\le m,~r\le l\le n$,
an $m\times n$ well-conditioned matrix $M$ 
of rank $r$, and a pair of  $k\times m$ and $n\times l$ Gaussian matrices $G$ and $H$, it holds that

(i) $GM$ is a left factor-Gaussian matrix of expected rank $r$ such that
$$||GM||_2\preceq ||M||_2~\nu_{k,r}~{\rm  
and}~||(GM)^+||_2\preceq ||M^+||_2~\nu_{k,r}^+,$$

(ii) $MH$ is a right factor-Gaussian matrix of expected  rank $r$ such that
%QL0419 || || -> || ||_2
$$||MH||_2 \preceq ||M||_2~\nu_{r,l}~{\rm 
and}~ 
%QL0419 || || -> || ||_2
||(MH)^+||_2\preceq ||M^+||_2~\nu_{r,l}^+,$$ 
%and
(iii) $GMH$ is a two-sided 
factor-Gaussian matrix of expected  rank 
$r$
such that $$||GMH||_2\preceq ||M||_2~\nu_{k,r}\nu_{r,l}~{\rm 
and}~||(GMH)^+||_2\preceq ||M^+||_2~\nu_{k,r}^+\nu_{r,l}^+.$$
\end{theorem}
\begin{remark}\label{reprprp}
Based on this theorem we can readily extend our results on $LRA$ of perturbed 
factor-Gaussian matrices to all matrices that admit $LRA$ and are pre-processed with Gaussian sketch matrices.
We cannot perform such pre-processing
at sublinear cost,
%VPN is not superfast,
 but empirically  pre-processing at sublinear cost with various Ultrasparse  sketch matrices having orthonormal columns tends to work as efficiently \cite{PLSZ16,
PLSZ17}.
\end{remark}

%------------------------------------------------------------------------------
%----------------------------------------------------
  
\section{The error bounds for  sketching algorithms}\label{slmmalg1}
 
%------------------------------------------------------------------------------
In the next theorem we write
$\sigma_{F,r+1}^2(M):=\sum_{j>r}\sigma_j^2(M)$.

 \begin{theorem}\label{therr51}
 %{eqgss1}
(i) Let $2\le r\le l-2$ and apply Alg.  \ref{alg1} with a Gaussian sketch matrix $H$. Then 
(cf. \cite[Thms. 10.5 and 10.6]{HMT11}) \footnote{\cite[Thms. 10.7 and 10.8]{HMT11} 
estimate the norms of $M-XY$
in probability.} 
%the norm  $|M-XY|$ 
%in probability.}
$$\mathbb E||M-XY||_F^2
 \le \Big(1+\frac{r}{l-r-1}\Big)~\sigma_{F,r+1}^2(M),$$
%  \end{equation}
%\begin{equation}\label{eqgss2}
$$\mathbb E||M-XY||_2\le \Big(1+\sqrt{\frac{r}{l-r-1}}~\Big)~
\sigma_{r+1}(M)+
 \frac{e\sqrt l}{l-r} \sigma_{F,r+1}(M).$$ 
 % \end{equation} 
 
%------------------------------------------------------------------------------

(ii) Let $4[\sqrt r+\sqrt{8\log(rn)}]^2\log(r)\le l\le n$ 
and apply Alg. \ref{alg1} with an SRHT or SRFT sketch matrix  $H$. Then 
(cf. \cite{T11}, \cite[Thm. 11.2]{HMT11}) 
 %  \begin{equation}\label{eqsrhft}
$$|||M-XY|||\le\sqrt{1+7n/l}~~
%QL0419 \bar -> \tilde
\tilde\sigma_{r+1}(M)~
{\rm with~a~probability}~1-O(1/r).$$ 
%\end{equation}
\end{theorem} 
 
\cite[Thm 4.3]{TYUC17} shows that the output {\em LRA} $XY$ of Alg. \ref{alg01} applied with 
 Gaussian sketch matrices $F$ and $H$ 
satisfies\footnote{In words,
% bound (\ref{eqtyuc}) shows  
the expected output error norm $\mathbb E||M-XY||_F$
 is within a factor of $\Big(\frac{kl}{(k-l)(l-r)}\Big)^{1/2}$ from its minimum value
$\sigma_{F,r+1}(M)$; this factor is just 2 for 
$k=2l=4r$.}
 \begin{equation}\label{eqtyuc}  
 \mathbb E||M-XY||_F^2\le 
 \frac{kl}{(k-l)(l-r)}\sigma_{F,r+1}^2(M)~{\rm if}~k>l>r.
\end{equation}   
 \begin{remark}\label{recwtyuc}
         Clarkson and Woodruff   
  prove in  \cite{CW09}  that Alg. \ref{alg01} reaches the bound 
$\sigma_{r+1}(M)$  
  within a factor of $1+\epsilon$ 
  {\em whp} if the sketch matrices 
  $F\in \mathcal G^{k\times m}$ and 
  $H\in \mathcal G^{n\times l}$ are Rademacher's matrices
  (filled with iid random variables, each equal to 1 or $-1$ with probability 1/2) and if $k$ and $l$ are sufficiently large, having order of  
  $r/\epsilon$ and $r/\epsilon^2$ for small 
  $\epsilon$, respectively,
although
 {\em LRA} is not practical  
  where the numbers $k$ and $l$ of row and column samples are  
    large (cf. \cite[Sec.
    1.7.3]{TYUC17}).
\end{remark}

%------------------------------------------------------------------------------

\section{Small families of hard inputs for sublinear cost {\em LRA}}\label{shrdin}

  Any sublinear cost {\em LRA} algorithm
  fails on the following small families of  inputs.
  
\begin{example}\label{exdlt} 
 Let  $\Delta_{i,j}$ denote an $m\times n$ matrix
 of rank 1  filled with 0s except for its $(i,j)$th entry filled with 1. 
Include the $m\times n$ null matrix $O_{m,n}$
filled with 0s  into the family of these $mn$ matrices.
  If an  {\em LRA} algorithm  does not involve the $(i,j)$th  
entry of its input matrix  for some pair of $i$ and $j$, as is the case for any sublinear cost algorithm,
 then it outputs the same approximation 
of the matrices $\Delta_{i,j}$ and $O_{m,n}$,
with an undetected  error at least 1/2.
Arrive at the same conclusion by applying the same argument to the
set of $mn+1$ small-norm perturbations of 
the matrices of the above family and to the                                                                                                                                   
 $mn+1$ sums  
of the latter matrices with  any
   fixed $m\times n$ matrix of low rank.
Finally, the same argument shows that 
a matrix norm estimator 
fails to produce even  reasonably close estimates for the norms of the matrices of the same $mn+1$ families
unless that  estimator 
involves all entries of an input matrix. 
\end{example}

This example can be extended to randomized  algorithms --   if for some pair $(i,j)$ an {\em LRA} algorithm or matrix norm estimator misses the $(i,j)$th  entry of an input matrix  with a  probability  $p$, then this algorithm
or estimator
fails with the probability $p$ on the above matrix families.

%----------------------------------------------------------- 

\section{Generation of two families of Ultrasparse sketch matrices}\label{spreprmlt}

%------------------------------------------------------------------------------

 %The two particular families of sparse sketch matrices of this sectionwere highly efficient in our tests, when we applied  them both themselves and in combination  with other sparse multipliers.  

%------------------------------------------------------------------------------

%\subsection{Generation of abridged Hadamard and %Fourier multipliers}\label{shad}

%------------------------------------------------------------------------------
 
We define  two families
of Ultrasparse sketch matrices 
  by  means of abridging 
    the classical 
recursive processes of the generation  of $n\times n$ 
SRHT and SRFT matrices
  for $n=2^t$. These matrices are 
  obtained from   the  $n\times n$ dense
matrices $H_n$ of {\em Walsh-Hadamard transform}
(cf. \cite[Sec. 3.1]{M11})
and $F_n$ of {\em discrete Fourier transform (DFT)} at $n$ points 
(cf. \cite[Sec. 2.3]{P01}),
respectively. Recursive representation 
 in $t$ recursive steps enables multiplication of the matrices $H_n$ 
 and $F_n$ by a vector by using $2tn$   additions and subtractions and
 $O(tn)$ flops, respectively.
 
 We end these processes in $d$ recursive steps for a
 fixed recursion depth $d$, $1\le d\le t$,
 and obtain the $d$-{\em abridged Hadamard (AH)  and Fourier
(AF)} matrices $H_{d,d}$
  and $F_{d,d}$, respectively, such that
  $H_{t,t}=H_n$ and  $F_{t,t}=F_n$. Namely,
   write $H_{d,0}:=F_{d,0}:=I_{n/2^d}$,
  let  $\omega_{s}:=\exp(2\pi \sqrt {-1}/s)$, denote  a primitive $s$-th root  of 1, and 
specify two recursive processes as follows: 
 
%------------------------------------------------------------------------------

\begin{equation}\label{eqrfd}
%H_{2}=H^{(2)}=
%\big (\begin{smallmatrix} 1  & ~~1  \\
%1   & -1\end{smallmatrix}\big )
H_{d,0}:=I_{n/2^d},~
H_{d,i+1}:=\begin{pmatrix}
H_{d,i} & H_{d,i} \\
H_{d,i} & -H_{d,i}
  \end{pmatrix}
  ~{\rm for}~i=0,1,\dots,d-1, 
%H_{2^{t-d}=\begin{pmatrix}
%I_{2^{t-d} & I_{2^{t-d}  \\
%I_{2^{t-d} & -I_{2^{t-d} 
%\end{pmatrix}
\end{equation}
\begin{equation}\label{eqfd}
F_{d,i+1}:=
\widehat P_{i+1}
\begin{pmatrix}F_{d,i}&~~F_{d,i}\\ 
F_{d,i}\widehat D_{i+1}&-F_{d,i}\widehat D_{i+1}\end{pmatrix},~
\widehat D_{i+1}:=\diag\Big (\omega_{2^{i+1}}^{j}\Big)_{j=0}^{2^i-1},~i=0,1,\dots,d-1.
\end{equation}
Here $\widehat P_{i}$ denotes the
$2^i\times 2^i$ matrix of odd/even permutations 
 such that 
$\widehat P_{i}{\bf u}={\bf v}$, ${\bf u}=(u_j)_{j=0}^{2^{i}-1}$, 
${\bf v}=(v_j)_{j=0}^{2^{i}-1}$, $v_j=u_{2j}$, $v_{j+2^{i-1}}=u_{2j+1}$, 
$j=0,1,\ldots,2^{i-1}-1$.\footnote{For $d=t$ this is a  decimation in frequency (DIF) radix-2 representation of FFT. 
Transposition turns it into the 
decimation 
in time (DIT) radix-2 representation of FFT.}
%In particular, 
%$$H_{1,1}=\begin{pmatrix}
%I_s  &  I_s  \\ 
%I_s  & -I_s  
%\end{pmatrix}~
%{\rm for}~s=n/2;~
%H_{2,2}=\begin{pmatrix}
%I_s  &  I_s & I_s  &  I_s \\
%I_s  & -I_s  & I_s  &  -I_s \\ 
%I_s  &  I_s & -I_s  &  -I_s \\
%I_s  & -I_s  & -I_s  &  I_s
%\end{pmatrix}~{\rm for}~s=n/4.$$
%$$H_{3,3}=\begin{pmatrix}
%I_s  &  I_s & I_s  &  I_s & I_s  &  I_s & I_s  &  %I_s \\
%I_s  & -I_s  & I_s  &  -I_s & I_s  & -I_s  & I_s  &%  -I_s\\ 
%I_s  &  I_s & -I_s  &  -I_s & I_s  & I_s  & -I_s  %&  -I_s\\
%I_s  & -I_s  & -I_s  &  I_s &  I_s  & -I_s  & -I_s  %&  I_s \\
%I_s  &  I_s & I_s  &  I_s & -I_s  &  -I_s & -I_s  %&  -I_s \\
%I_s  & -I_s  & I_s  &  -I_s & -I_s  &  I_s & -I_s  %&  I_s\\ 
%I_s  &  I_s & -I_s  &  -I_s & -I_s  &  -I_s & I_s  %&  I_s\\
%I_s  & -I_s  & -I_s  &  I_s & -I_s  &  I_s & I_s  %&  -I_s
%\end{pmatrix}~{\rm for}~s=n/8.$$

For a fixed pair of $d$ and $i$, 
 the matrix  
 $H_{d,i}$ (resp. $F_{d,i}$) 
has orthonormal columns (and hence is orthogonal (resp. unitary) since it is  a square matrix) 
up to scaling and
 has $2^d$ nonzero entries 
in every row and  column.  
Now make up  sketch matrices $F$
and  $H$  of  $k\times m$
and  $n\times l$ submatrices of 
$F_{d,d}$ and $H_{d,d}$, respectively. Then
in view  of sparseness of $F_{d,d}$ or $H_{d,d}$, 
we can compute the products 
$FM$ and $MH$ by using 
$O(kn2^d)$ and $O(lm2^d)$ flops, respectively. %which are just additions or subtractions in the case of submatrices of $H_{d,d}$.

 %\medskip

Define the 
%families of 
$d$--{\em Abridged Scaled and Permuted 
 Hadamard (ASPH)} matrices, $PDH_{d,d}$, and 
$d$--{\em Abridged Scaled and Permuted 
Fourier  (ASPF)} $n\times n$
 matrices, $PD'F_{d,d}$, where $P$ is a
 random 
 sampling %V%%,
 matrix, $D$ is the matrix of  
  Rademacher's  or another random integer diagonal scaling, and $D'$ is a matrix
 of random unitary diagonal scaling.
Likewise  define the families of
ASH, ASF, APH, and APF matrices, 
 $DH_{n,d}$, $DF_{n,d}$, $H_{n,d}P$, and  $F_{n,d}P$, respectively.
Each random permutation or scaling  
 contributes up to $n$ random parameters.  
We can involve more random parameters by applying random permutation and scaling 
also to some or all
 intermediate matrices $H_{d,i}$
and  $F_{d,i}$ for $i=0,1,\dots,d$.

The first $k$ rows for
 $r\le k\le n$  or first $l$ columns for  
 $r\le l\le n$ of $H_{d,d}$ and 
 $F_{d,d}$ form a
 $d$-abridged Hadamard or Fourier sketch matrix, which turns into
a SRHT or SRFT matrix,
respectively, for $d=t$.
  For $k$ and $l$   of order $r\log(r)$ 
   Alg. \ref{alg1} with a SRHT or SRFT sketch matrix outputs  whp accurate {\em LRA}
 of any matrix $M$ admitting {\em LRA} 
 (see \cite[Sec. 11]{HMT11}), 
  but in our tests the output was  
% consistently 
 accurate even
with Ultrasparse abridged SRHT or SRFT sketch matrices,  even when we computed them just in three recursive steps and added a couple of
abridged matrices of random permutation (see Sec. \ref{srndsmpl}).

%------------------------------------------------------------------------------
   
\medskip
\medskip
%------------------------------------------------------------------------------

\noindent %{\bf Acknowledgements:}Our research has been %supported by NSF Grants  CCF--1563942 and CCF--1733834 and PSC CUNY Award 69813 00 48. 
% We also thank E.E. Tyrtyshnikov for the challenge of providing formal support for  empirical efficiency of {\em C-A} iterations, and we  very much appreciate reviewers' thoughtful comments, which helped us to improve our initial draft most significantly.

%N. L. Zamarashkin pointed us out reference \cite{O18}.

%------------------------------------------------------------------------------

%------------------------------------------------------------------------------


\begin{thebibliography}{hspace{0.5in}}

%--------------------------------------------------------------------------------
%---------------------------------------------------------

\bibitem{ALS24}
Kenneth Allen, Ming-Jun Lai,  Zhaiming Shen,
Maximal Volume Matrix Cross Approximation for Image
 Compression and Least Squares Solution,  {\em Advances in Computational Mathematics}, {\bf 5}, 2024. DOI: 10.1007/s10444-024-10196-7.
Also arXiv:2309.1740   (December, 2024).

%--------------------------------------------------------------------------------


%------------------------------------------------------------------------------

\bibitem{B15}
A. Bj{\"o}rk, {\em Numerical Methods in Matrix Computations}, Springer, New York, 2015.

%------------------------------------------------------------------------------

\bibitem{BV88}
W. Bruns, U. Vetter, 
Determinantal Rings, 
Lect. Notes Math., {\bf 1327} ,
 Springer, 1988.

%--------------------------------------------------------------------------------

\bibitem{BW18}
A. Bakshi, D. P. Woodruff: Sublinear Time Low-Rank Approximation of
Distance Matrices, {\em Procs.  32nd Intern. Conf. Neural Information Processing Systems (NIPS'18),} 3786--3796, Montréal, Canada, 2018.


%\bibitem[BW17]{BW17}
%C. Boutsidis, D. Woodruff,
%Optimal CUR Matrix Decompositions,
%{\em SIAM Journal on Computing},
%{\bf 46,~2}, 543--589, 2017,
%DOI:10.1137/140977898.
%Also see arXiv:1405.7910 

%-----------------------------------------------------------------------

\bibitem{C16}
Michael B. Cohen, Nearly tight oblivious subspace embeddings by trace inequalities,
{\em 27th ACM-SIAM Symp. on Discrete Algorithms (SODA 2016)}, Arlington,  278 --287,  2016.
doi:10.1137/1.9781611974331.ch21.8

%-----------------------
%-----------------------------------------------------------------------

\bibitem{CD05}
Z. Chen, J. J. Dongarra, Condition Numbers of Gaussian Random Matrices,
{\em SIAM. J. on Matrix Analysis and Applications}, {\bf 27}, 603--620, 2005.

%---------------------------------------------------------------------
%-----------------------------------------------------------------------

\bibitem{CD13}
Jiawei Chiu, Laurent Demanet, Sublinear randomized algorithms for skeleton decompositions, 
{\em SIAM J. Matrix Anal. Appl.}, {\bf 34}, 1361--1383,  2013.\\ https://doi.org/10.1137/110852310. Also
arXiv 1110.4193 (Oct 2011).

%-----------------------------------------------------------------------

\bibitem{CDDRa}
Shabarish Chenakkod, Michał Derezi´nski, Xiaoyu Dong,  Mark Rudelson, Optimal embedding dimension for sparse subspace embeddings,  arXiv:2311.10680 (2023, revised June 2024).

%-----------------------------------------------------------------------
%-----------------------------------------------------------------------

\bibitem{CETW23}
Yifan Chen, Ethan N. Epperly,
Joel A. Tropp, Robert J. Webber,
Randomly pivoted Cholesky:
Practical approximation of a kernel matrix
with few entry evaluations, arXiv 2207.06503  
(December 2023,  
last revised 22 Oct 2024).

%--------------------------------------------------------------------------------

%--------------------------------------------------------------------------------

\bibitem{CFSa}
Coralia Cartis, Jan Fiala,  Zhen Shao, Hashing embeddings of optimal dimension, with applications to linear least squares, arXiv:2105.11815 (May 25, 2021).

%--------------------------------------------------------------------------------
%--------------------------------------------------------------------------------

\bibitem{CK20}
Alice Cortinovis,
Daniel  Kressner, Low-Rank Approximation in the Frobenius Norm by Column and Row Subset Selection, {\em SIAM Journal on Matrix Analysis and Applications}, {\bf 41, 4}, 1651-1673, 2020. Also
arXiv:
(Aug 16, 2019).

%--------------------------------------------------------------------------------

%{ccivril2009selecting}
\bibitem{CM-I09}
 A. \c{C}ivril,  M. Magdon-Ismail, 
On selecting a maximum volume sub-matrix of a matrix and related problems. 
Theor. Computer Sci. 
 {\bf 410(47-49)}, 4801 -- 4811, 2009.

%--------------------------------------------------------------------------------

\bibitem{CW09}
K. L. Clarkson, D. P. Woodruff, Numerical Linear Algebra in the Streaming Model, 
{\em  ACM Symp. Theory of Computing (STOC 2009)}, 205--214, ACM Press, NY, 2009.

%------------------------------------------------------------------------------

\bibitem{CY25} Alice Cortinovis, Lexing Ying, A Sublinear-Time Randomized Algorithm for Column and Row Subset Selection Based on Strong Rank-Revealing QR Factorizations, {\em SIAM Journal on Matrix Analysis and Applications},
 {\bf 46, 1}, 22-44, 2025.\\
   https://doi.org/10.1137/24M164063X. 
 Also arXiv 2402.13975 (February 2024).

%------------------------------------------------------------------------------

%------------------------------------------------------------------------------

%\bibitem[D88]{D88}
%J. Demmel,  
%The Probability That a Numerical Analysis Problem Is Difficult,
%{\em Math. of Computation}, {\bf 50}, 449--480, 1988.

%------------------------------------------------------------------------------

%\bibitem{DMM08}P. Drineas, M.W. Mahoney, S. Muthukrishnan, Relative-error CUR Matrix Decompositions, {\em SIAM Journal on Matrix Analysis and Applications}, {\bf 30,~2}, 844--881,  2008.

%--------------------------------------------------------------------------------

\bibitem{DS01}
K. R. Davidson, S. J. Szarek, 
Local Operator Theory, Random Matrices, and Banach Spaces, 
in {\em Handbook on the Geometry of  Banach Spaces} 
(W. B. Johnson and J. Lindenstrauss editors), pages 317--368, 
North Holland, Amsterdam, 2001. 

%------------------------------------------------------------------------------

\bibitem{E89}
A. Edelman,
Eigenvalues and Condition Numbers of Random Matrices, Ph.D. thesis,
Massachusetts Institute of Technology, 1989.
%------------------------------------------------------------------------------
%------------------------------------------------------------------------------

\bibitem{ES05}
A. Edelman, B. D. Sutton,  Tails of Condition Number Distributions,
{\em SIAM J. on Matrix Analysis and Applications}, {\bf 27}, {\bf 2},
547--560, 2005.
 
%------------------------------------------------------------------------------

\bibitem{GE96}
M. Gu, S.C. Eisenstat, 
An Efficient Algorithm for Computing a Strong Rank Revealing QR Factorization, 
{\em SIAM J. Sci. Comput.}, {\bf 17}, 848--869, 1996.

%------------------------------------------------------------------------------
%------------------------------------------------------------------------------

\bibitem{GT16}
Gratton, Serge, J. Tshimanga‐Ilunga. 
On a second‐order expansion of the truncated singular subspace decomposition,
{\em Numerical Linear Algebra with Applications},{\bf 23}, {\bf 3}, 519-534, 2016.

%------------------------------------------------------------------------------
%--------------------------------------------------------------------------------

\bibitem{GL13}
G. H. Golub, C. F. Van Loan, 
{\em Matrix Computations},
The Johns Hopkins University Press, Baltimore, Maryland, 2013 (fourth edition).

%------------------------------------------------------------------------------

\bibitem{HMT11}
N. Halko, P. G. Martinsson, J. A. Tropp,
Finding Structure with Randomness: Probabilistic Algorithms
for Constructing
 Approximate Matrix Decompositions, 
{\em SIAM Review}, {\bf 53,~2}, 217--288, 2011.

%1------------------------------------------------------------------------------

\bibitem{L09}
E. Liberty,
Accelerated Dense Random Projections, 
PhD Thesis, Yale Univ., 2009. 

%1------------------------------------------------------------------------------

%\bibitem[LAS11]{LAS11}
%E. Liberty, N. Ailon, A. Singer,
%Dense Fast Random Projections and Lean Walsh %Transforms, 
%{\em J. of Discrete  and Computational Geometry},
%{\bf 45, 1}, 34--44, 2011. 

%------------------------------------------------------------------------------

\bibitem{LM00}
B. Laurent, P. Massart,
Adaptive estimation of a quadraticfunctional by model selection,
{\em Annals of Statistics}, 1302--1338, 2000.

%------------------------------------------------------------------------------
%------------------------------------------------------------------------------

\bibitem{LP20}
Q. Luan, V. Y. Pan,
CUR LRA at Sublinear Cost Based on Volume Maximization,
%VP
  LNCS 11989,  
In Book: Mathematical Aspects of Computer and Information Sciences
(MACIS 2019), D. Salmanig et al (Eds.),
 Springer Nature Switzerland AG 2020,
Chapter No: 10, pages 1-- 17, 
Chapter DOI:10.1007/978-3-030-43120-4\_10

 http://doi.org/10.1007/978-3-030-43120-4\_9
 %VP
and arXiv:1907.10481, (July 21, 2019).

%-----------------------------------------------------------------------------

%\bibitem{LPSa} Q. Luan, V. Y. Pan,  J. Svadlenka,The Dual JL Transforms and Superfast Matrix Algorithms,preprint in arXiv:1906.04929 (April 3 2021).


\bibitem{LYMHY}
Y. Li, H. Yang, E. R. Martin, K. L. Ho,  L. Ying, Butterfly factorization, {\em Multiscale
Model. Simul.}, {\bf 13},  714--732, 2015. https://doi.org/10.1137/15M1007173.
 Also arXiv:1502.01379
(February 2015).

%-----------------------------------------------------------------------------
%-----------------------------------------------------------------------------

\bibitem{M11}
M. W. Mahoney,
Randomized Algorithms for Matrices and Data,  
{\em Foundations and Trends in Machine Learning}, NOW Publishers, {\bf 3,~2}, 2011. Preprint: arXiv:1104.5557 (2011)
(Abridged version in: {\em Advances in Machine Learning and Data Mining for Astronomy}, 
edited by M. J. Way et al., pp. 647--672, 2012.) 

%-----------------------------------------------------------------------------
%-----------------------------------------------------------------------

%\bibitem{MD09}M. W. Mahoney, P. Drineas,CUR Matrix Decompositions for Improved Data Analysis, {\em Proc. Natl. Acad. Sci. USA}, {\bf 106}, 697 -- 702, 2009.

%-----------------------------------------------------------------------

\bibitem{MT20} 
Per-Gunnar Martinsson, Joel A. Tropp, Randomized numerical linear algebra: Foundations and algorithms, {\em Acta Numerica}, {\bf 29,} 403--572  (2020).
 
%----------------------------------------------------------------------- 
%-----------------------------------------------------------------------

\bibitem{MW17}
Cameron Musco, D. P. Woodruff: Sublinear Time Low-Rank Approximation of
Positive Semidefinite Matrices, 
{\em IEEE 58th Annual Symposium on Foundations of Computer Science (FOCS),}
 672--683, 2017.

%------------------------------------------------------------------------------

%------------------------------------------------------------------------------

\bibitem{MZ08}
A. Magen, A. Zouzias,
Near optimal dimensionality reductionsthat preserve volumes,
{\em Approximation, Randomization and Combinatorial Optimization. Algorithms and Techniques},
523 -- 534, 2008.

%------------------------------------------------------------------------------
%------------------------------------------------------------------------------

\bibitem{N20}
Yuji Nakatsukasa,
Fast and stable randomized low-rank matrix approximation,
arXiv:2009.11392 (Sept 2020)

%--------------------
%-----------------------------------------------------------------------

%\bibitem[O17]{O17}
%A. Osinsky,
%Probabilistic estimation of the rank 1 
%cross approximation accuracy,
% June 2017, arXiv:1809.02334 



%------------------------------------------------------------------------------

%\bibitem[OZ16]{OZ16}
%A.I. Osinsky, N. L. Zamarashkin, 
%New Accuracy Estimates for
%Pseudo-skeleton Approximations
%of Matrices, Doklady Mathematics, 
%{\bf 94,~3}, 643--645, 2016.
 
%-----------------------------------------------------------------------

\bibitem{OZ18}
A.I. Osinsky, N. L. Zamarashkin,
 Pseudo-skeleton Approximations
 with Better Accuracy Estimates,
{\em Linear Algebra  Applics.}, {\bf 537}, 221--249, 2018.
  
%------------------------------------------------------------------------------

\bibitem{P00}
C.-T. Pan,
On the Existence and Computation of
Rank-Revealing LU Factorizations,
{\em Linear Algebra and Its Applications}, {\bf 316}, 199--222, 2000.

%------------------------------------------------------------------------------

\bibitem{P01}
V. Y. Pan,
{\em Structured Matrices and Polynomials: Unified Superfast Algorithms}, Birk\-h\"auser/Sprin\-ger, Boston/New York, 2001. 

%------------------------------------------------------------------------------

\bibitem{P15}
V.Y. Pan, 
Transformations of Matrix Structures Work Again. 
Linear Algebra and Its Applications, {\bf 465}, 
1-32, 2015.
doi: 10.1016/j.laa.2014.09.004

%------------------------------------------------------------------------------


%------------------------------------------------------------------------------ 

%\bibitem{Pa} V. Y. Pan,  Low Rank Approximation of a Matrix at Sub-linear Cost, \\arXiv:1907.10481, 21 July 2019.  
 
%------------------------------------------------------------------------------

%\bibitem{PLa} V. Y. Pan, Q. Luan,  Refinement of Low Rank Approximation of a Matrix at Sub-linear Cost, arXiv:1906.04223 (Submitted on 10 Jun 2019).

%------------------------------------------------------------------------------

\bibitem{PLSZ16}
V. Y. Pan, Q. Luan, 
J. Svadlenka, L. Zhao,
Primitive and Cynical Low Rank Approximation, Preprocessing and Extensions,
arXiv 1611.01391 (3 November,
2016).

%------------------------------------------------------------------------------

\bibitem{PLSZ17}
V. Y. Pan, Q. Luan, J. Svadlenka, L. Zhao,
Superfast Accurate Low Rank Approximation, preprint,
arXiv:1710.07946 (22 October, 
2017).

%------------------------------------------------------------------------------

\bibitem{PLSZa}
V. Y. Pan, Q. Luan, J. Svadlenka, L. Zhao,
CUR Low Rank Approximation
at Sub-linear Cost, arXiv:1906.04112 
(Revised in 2020).

%------------------------------------------------------------------------------

\bibitem{PLSZb}
V. Y. Pan, Q. Luan, J. Svadlenka, L. Zhao,
Low Rank Approximation at Sub-linear Cost by Means of Subspace Sampling, arXiv:1906.04327  (Submitted on 10 Jun 2019).

%------------------------------------------------------------------------------
%------------------------------------------------------------------------------

\bibitem{PQY15} 
V. Y. Pan, G. Qian, X. Yan, 
Random Multipliers Numerically Stabilize Gaussian and Block Gaussian Elimination: 
Proofs and an Extension to Low-rank Approximation,
{\em Linear Algebra and Its Applications}, {\bf 481}, 202--234, 2015.

%-----------------------------------------------------------------------------

\bibitem{PZ17}
V. Y. Pan, L. Zhao,
Numerically Safe Gaussian Elimination
with No Pivoting,
{\em  Linear Algebra and Its Applications}, {\bf 527}, 
349--383, 2017.

%----------------------------------------------------------------------------

\bibitem{RV09}
M. Rudelson, R. Vershynin, Smallest 
Singular Value of a Random Rectangular Matrix, {\em Comm. Pure Appl. Math.}, {\bf 62,~12}, 1707--1739, 2009.  \\
https://
doi.org/10.1002/cpa.20294

%----------------------------------------------------------------

\bibitem{S73}
G. W. Stewart,
Error and Perturbation Bounds for Subspaces Associated with Certain Eigenvalue Problems,
SIAM Review, {\bf 15 (4)}, 727--764, 1973. \\ https://doi.org/10.1137/1015095

%------------------------------------------------------------------------------

%\bibitem{S06}T. Sarl\'os, Improved Approximation Algorithms for Large Matrices viaRandom Projections, {\em FOCS'06}, 143--152, 2006.

%------------------------------------------------------------------------------

\bibitem{S16}
V. Simoncini,   Computational Methods for Linear Matrix Equations, {\em SIAM Review}, {\bf 58 (3)} 377 -- 441, 2016.
doi:10.1137/130912839

%------------------------------------------------------------------------------
%------------------------------------------------------------------------------

\bibitem{SST06}
A. Sankar, D. Spielman, S.-H. Teng, 
Smoothed Analysis of the Condition Numbers and Growth Factors of Matrices, 
 {\em SIAM J. Matrix Anal. Appl.}, {\bf 28}, {\bf 2}, 446--476, 2006. 

%6------------------------------------------------------------------------------

\bibitem{T11}
J. A. Tropp,
Improved analysis of the subsampled randomized 
%VP
Hadamard 
%VP
transform,
{\em Adv. Adapt. Data Anal.}, {\bf 3,~1--2} 
(Special issue "Sparse Representation of Data and Images"), 115--126, 2011.
Also arXiv:  1011.1595
(6 Nov 2010).

%6------------------------------------------------------------------------------

%QL0127 
%new reference type Book
\bibitem{T12}
Y.L. Tong, 
The multivariate normal distribution,
Springer Science \& Business Media, 2012.

%6---------------------------------------------------------------

\bibitem{TWa}
Joel A. Tropp, Robert J. Webber, Randomized algorithms for low-rank matrix approximation: Design, analysis, and applications,  arXiv:
2306.12418 (June 21, 2023).

%------------------------------------------------------------------------------

%-------------------------------------------------------------------------------

\bibitem{TYUC17}
J. A. Tropp, A. Yurtsever, M. Udell, V. Cevher, Practical Sketching Algorithms for Low-rank Matrix Approximation, 
{\em SIAM J. Matrix Anal. Appl.}, {\bf 38},
%, ~4}
1454--1485, 2017.


%6------------------------------------------------------------------------------

\bibitem{TYUC19} 
J. A. Tropp, A. Yurtsever, M. Udell, and V. Cevher, Streaming Low-Rank Matrix Approximation with an Application to Scientific Simulation,
SIAM J. Sci. Comp., {\bf 41, 4},  A2430-A2463, 2019. doi:10.1137/18M1201068.
arXiv: 1902.08651 
(Feb 22, 2019). 

%------------------------------------------------------------------------------

\bibitem{UT19}
M. Udell, A. Townsend, Why are big data matrices approximately of low rank?, {\em SIAM J. Math. Data Sci.}, {\bf 1}, 144-160, 2019.

%------------------------------------------------------------------------------

\bibitem{W72}
Per-Åke Wedin, Perturbation bounds in connection with the singular value decomposition, {\em BIT Numerical Mathematics,} {\bf 12}, 99-111, 1972.

%------------------------------------------------------------------------------
%------------------------------------------------------------------------------

\bibitem{W14}
D. P. Woodruff, {\em Sketching as a tool for numerical linear algebra},  Found. Trends in Theor.
Comput. Sci., {\bf 10}, pp. iv+157, 2014.

%------------------------------------------------------------------------------

\bibitem{X24}
J. Xia, Making the Nystr\"om method highly accurate for low-rank approximations, {\em SIAM J. Sci.
Comput.}, {\bf 46}, A1076--A1101  2024.
\\  https://doi.org/10.1137/23M1585039.
Also arXiv:2307.05785 (July 2023).

%------------------------------------------------------------------------------
%------------------------------------------------------------------------------

\bibitem{XXG12} 
J. Xia, Y. Xi, M. Gu, A superfast structured solver for Toeplitz linear systems via randomized sampling,
{\em SIAM J. Matrix Anal. Appl.}, {\bf 33},  837--858,  2012.  
 
%------------------------------------------------------------------------------

\end{thebibliography}
\end{document}